\documentclass[12pt,oneside]{amsart}

\usepackage{amsmath}
\usepackage{amssymb}  
\usepackage{mathrsfs}
\usepackage{pifont}
\usepackage{color,xcolor}
\usepackage{cite}
\usepackage{graphicx}
\usepackage{wrapfig}
\usepackage{indentfirst}
\usepackage{verbatim}
\usepackage{makecell}
\usepackage{subfigure}
\usepackage{float}
\usepackage{amsthm}
\usepackage{enumerate}
\usepackage{bbm}
\usepackage{mathtools}
\usepackage{geometry}
\usepackage{appendix}
\usepackage{tcolorbox}
\usepackage{pgfplots}
\usepackage{tikz}
\usepackage{tikz-network}

\usepackage[
pdfstartview=FitH,
CJKbookmarks=true,
bookmarksnumbered=true,
bookmarksopen=true,
colorlinks=true,
citecolor=magenta,
linkcolor=blue,
]{hyperref}

\newtheorem{theorem}{Theorem}[section]

\newtheorem{corollary}[theorem]{Corollary}
\newtheorem{proposition}[theorem]{Proposition}

\theoremstyle{definition}
\newtheorem{definition}[theorem]{Definition}
\theoremstyle{remark}
\newtheorem*{remark}{Remark}

\numberwithin{equation}{section}

\title[Spectral concentration for damped waves]{The spectral concentration for damped waves on compact Anosov manifolds}
\author{Yulin Gong}
\address{Department of Mathematical Sciences, Tsinghua University, Beijing, China.}
\email{gongyl22@mails.tsinghua.edu.cn}
\date{}
\begin{document}
\maketitle
\begin{abstract}
  We study the spectral distribution of damped waves on compact Anosov manifolds. Sjöstrand \cite{SJ1} proved that the imaginary parts of the majority of the eigenvalues concentrate near the average of the damping function, see also Anantharaman \cite{AN2}. In this paper, we prove that the most of eigenvalues actually lie in certain regions with imaginary parts that approach the average logarithmically as the real parts tend to infinity. The proof relies on the moderate deviation principles for Anosov geodesic flows. As an application, we show the concentration of non-trivial zeros of twisted Selberg zeta functions in a logarithmic region asymptotically close to $\Re s=\frac{1}{2}$.
\end{abstract}
\section{Introduction}
Let $(M, g)$ be a smooth, connected, compact Riemannian manifold without boundary, and $a \in C^{\infty}(M,\mathbb{R})$\footnote{Here, we refer to $a$ as the damping function. However, we do not assume the sign of $a$ in this paper.}. We study the ``damped wave equation" given by
\begin{equation}\label{DWEeq}
\left(\partial_t^2 - \Delta + 2a(x) \partial_t\right) v = 0, \quad t \in \mathbb{R}, \quad x \in M.
\end{equation}
We are interested in the stationary solutions $v(t, x) = e^{-i t \tau} u(x)$ for some $\tau \in \mathbb{C}$. This implies that $u$ must satisfy
\begin{equation}\label{DWEspec}
\left(-\Delta - \tau^2 - 2i a \tau\right) u = 0.
\end{equation}
\noindent Equivalently, $\tau$ is an eigenvalue of the operator
\begin{equation}\label{DWEop}
\left(\begin{array}{cc}
0 & I \\
-\Delta & -2i a
\end{array}\right)\colon H^2(M)\times H^{1}(M) \subset H^1(M) \times L^2(M) \mapsto H^1(M) \times L^2(M).
\end{equation}
It is known that spectrum consists of a discrete sequence of eigenvaluesc$\{\tau_n\}$ in $\mathbb{C}$ with $\Im \tau_n$ bounded and $|\Re \tau_n| \rightarrow \infty$. The spectrum is symmetric with respect to the imaginary axis, i.e., $-\bar{\tau}$ is an eigenvalue if $\tau$ is; see, e.g., Lebeau \cite{LB96}.

Our main theorem shows the spectral concentration region for damped waves on compact Anosov manifolds (see Definition \ref{anosov}) as follows:
\begin{theorem}\label{concentrationforDWE}
  Assume that $M$ is a $d$-dimensional Anosov compact manifold, $a \in C^{\infty}(M, \mathbb{R})$, and $\{\tau_n\}$ is the spectrum of the operator \eqref{DWEop}. Let $\overline{a} = \frac{1}{\mathrm{Vol}(M)}\int_{M}a(x)\,d\mathrm{Vol}(x)$ be the average of the damping function $a(x)$ on $M$, there exists a constant $c(a,M) \in (0,\infty]$ which is defined in \eqref{eq82} such that for any $0<c<c(a, M)$, and any $0<\alpha<1$, we have
\begin{equation}\label{eq15}
\#\left\{n \ \middle|\ 0 \leq \Re \tau_n \leq \lambda, \ \left|\Im \tau_n + \overline{a}\right|\geq (\log \lambda)^{-\frac{1-\alpha}{2}} \right\}=\mathcal{O}\left(\frac{\lambda^{d}}{e^{c|\log \lambda|^{\alpha}}(\log \lambda)^{\alpha-1}}\right), \quad \lambda \to \infty.
\end{equation}
\end{theorem}

\begin{figure}
  \centering
 \begin{tikzpicture}
\begin{axis}[xmin=5,xmax=12,ymin=-0.4,ymax=0.05,
  axis lines = middle,
  ytick={},
  yticklabels={},
  xtick={},
  xticklabels={},
  xlabel = \(\Re \tau\),
  ylabel = {\(\Im \tau\)},
]
\addplot[
  domain=5.5:12,
  samples=1000,
  color=red]{-0.2};

\addplot[
  domain=5.5:12,
  samples=1000,
  color=blue]{ 0.7*(ln x)^(-5)-0.2};
\addplot[
    domain=5.5:12,
    samples=1000,
    color=blue]{ -0.7*(ln x)^(-5)-0.2};
    \addplot[
      domain=5.5:12,
      samples=1000,
      color=green]{ -0.14};
      \addplot[
        domain=5.5:12,
        samples=1000,
        color=green]{ -0.26};
    
\end{axis}
\end{tikzpicture}
  \caption{Spectral concentration region: The red line marks the average of the damping term, the blue line marks the boundary of a logarithmic region asymptotically approaching the average, and the green line marks the boundary of a horizontal strip around the average.}
  \label{CRFS}
\end{figure}
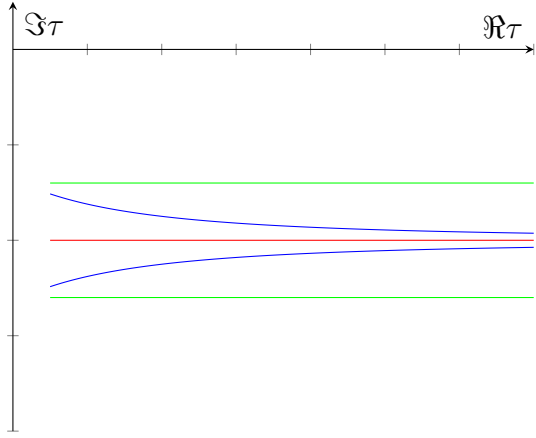
\noindent In our proof, we need the moderate deviation principles (MDP) for compact Anosov manifolds, see Theorem \ref{MDPA} in Section \ref{sec2}.

Recall that Markus and Matsaev \cite{MAMV} proved an analogue of Weyl law, which was also proved independently by Sjöstrand \cite{SJ1}:
\begin{theorem}[Markus and Matsaev \cite{MAMV}, Sjöstrand \cite{SJ1}]\label{renumber}
Let $(M, g)$ be a $d$-dimensional smooth, compact Riemannian manifold without boundary, and $a \in C^{\infty}(M, \mathbb{R})$. Then we have:
\begin{equation}\label{eq24}
\#\left\{n \ \middle|\ 0 \leq \Re \tau_n \leq \lambda\right\}=\left(\frac{\lambda}{2 \pi}\right)^d\left(\int_{|\xi|_{x}\leq 1} d x d \xi+\mathcal{O}\left(\lambda^{-1}\right)\right), \quad \lambda \to \infty.
\end{equation}
\end{theorem}

Combining Theorems \ref{concentrationforDWE} and \ref{renumber}, we show that on the compact Anosov manifold $M$, as the real part of eigenvalues $\Re \tau \to \infty$, most of the eigenvalues are distributed within a logarithmic region (see Figure \ref{CRFS}) around $-\overline{a}$. The region is asymptotically close to $\Im \tau = -\overline{a}$ as the frequency $\Re \tau \to \infty$.

When the geodesic flow is ergodic with respect to the normalized Liouville measure $\mu_{L}$ on the sphere bundle $S^*M$, Sjöstrand \cite{SJ1} first proved the spectral concentration of damped waves around the average $-\overline{a}$: for every $\epsilon > 0$, we have 
\begin{equation}\label{eq14} 
  \# \left\{n \ \middle| \ 0 \leq \Re \tau_n \leq \lambda, \ \left|\Im \tau_n + \overline{a}\right|\geq \epsilon \right\} = o\left(\lambda^{d}\right), \quad \lambda \to \infty. 
\end{equation} 
Furthermore, if $M$ is a compact Anosov manifold, Anantharaman \cite{AN2} improved the upper bound in \eqref{eq14}: for every $\epsilon > 0$, there exists a positive constant $c = c(a, \epsilon) > 0$ such that 
\begin{equation}\label{eq30} 
  \# \left\{n \ \middle| \ 0 \leq \Re \tau_n \leq \lambda, \ \left|\Im \tau_n + \overline{a}\right|\geq \epsilon \right\} = \mathcal{O}\left(\lambda^{d - c}\right), \quad \lambda \to \infty. 
\end{equation} 
Here, the exponent $c$ is determined by the rate function for large deviation principles (LDP); see \cite{AN2, LSY, KY}. Both \eqref{eq14} and \eqref{eq30} show that most of the eigenvalues are distributed within a horizontal strip around $-\overline{a}$.

\subsection{Semiclassical reduction} 
A semiclassical formulation of the problem was used in \cite{SJ1}; see also \cite{AN2}. We introduce the semiclassical calculus in Section \ref{sec2}. Take any $0 < h \ll 1$. To study the frequencies $\Re \tau = h^{-1} + \mathcal{O}(1)$, let $z = \frac{h^2 \tau^2}{2}$. The equation \eqref{DWEspec} then becomes $(P(z, h) - z) u = 0$, where $P(z, h) = -\frac{h^2 \Delta}{2} - ih\sqrt{2z} a$. More generally, we consider the following semiclassical pseudodifferential operator:
\begin{equation}\label{eq7}
P(z, h) = P + ih Q(z), \quad P \in \Psi^{2}(M), \ Q(z) \in \Psi^{1}(M),
\end{equation}
where $P$ is formally self-adjoint with principal symbol $\sigma_{h}(P) = \frac{1}{2}|\xi|_{x}^2$, and $Q(z)$ depends holomorphically on $z \in D\left(\frac{1}{2}, \frac{1}{2}\right)$\footnote{Here, $D(a, r) \subset \mathbb{C}$ denotes the disk centered at $a$ with radius $r$}, and $Q(z)$ is formally self-adjoint if $z \in \mathbb{R}$; see \cite{SJ1}. We call $P(z,h)$ semiclassical damped wave operator. We denote the set of the poles of $(z-P(z,h))^{-1}$
$$\Sigma = \left\{z \in D\left(\frac{1}{2}, \frac{1}{2}\right) \ | \ \exists u \in L^2(M), \ (P(z, h) - z) u = 0\right\}.$$
We fix some $C > 0$ and denote
\begin{equation}\label{eq9}
\Sigma_{\frac{1}{2}} := \left\{z \in \Sigma \ \Bigg| \ \frac{1}{2} - Ch \leq \Re z \leq \frac{1}{2} + Ch \right\}.
\end{equation}
We denote $q_{z}(x, \xi) := \sigma_{h}(Q(z))$ as the principal symbol of $Q(z)$, and $q(x, \xi) := q_{\frac{1}{2}}(x, \xi)$ is called the damping term, and recall that $\mu_{L}$ is the normalized Liouville measure on $S^*M$. Our main result is the following theorems:

\begin{theorem}\label{mainthm1}
  Let $M$ be a compact Anosov manifold of dimension $d$, and let $\overline{q} = \int_{S^*M} q(\rho)\, d\mu_{L}(\rho)$. Then there exists a constant 
  \begin{equation}\label{eq82}
  c(q,M) = \frac{1}{2\Lambda_{0}\sigma_{q}^2} \in (0,\infty],
  \end{equation}
  such that for any $0 < c < c(q,M)$ and any width function $w(h) > 0$ satisfying $\lim\limits_{h\to 0} w(h) = 0$ and $\lim\limits_{h\to 0} w^2(h) |\log h| = \infty$, we have
  \begin{equation}\label{eq10}
  \# \left\{ z \in \Sigma_{\frac{1}{2}} \, \Big| \, \left| \frac{\Im z}{h} - \overline{q}\right| \geq w(h) \right\} = \mathcal{O}\left(\frac{h^{1-d}}{e^{cw(h)^2|\log h|}w(h)^2}\right).
  \end{equation}
  Here, $\sigma_{q}^2$ is defined in \eqref{limitofvariance} as the limit of the variance of $\frac{1}{\sqrt{t}} \int_{0}^{t} q \circ \varphi^{s}ds$ on the probability space $(S^*M, \mu_{L})$ as $t \to \infty$.\footnote{This result can be found in Ratner \cite[Theorem 3.1]{RM73}, and if $\sigma_{q}=0$, we define $c(q,M)=\infty$.} The constant $\Lambda_{0}>0$, given in \eqref{maximalexpansion}, is the maximal expansion rate of the geodesic flow $\varphi^{t}$ on $S^*M$.
  \end{theorem}

Theorem \ref{mainthm1} shows that the width of concentration for the spectrum of damped waves is given by $w(h) \sim |\log h|^{-\frac{1-\alpha}{2}}$ for any \(\alpha > 0\). A similar result was previously established by Schenck in \cite[Theorem 1.3]{SE09} for the spectrum of damped quantum maps, where he proved that the width of concentration is also $w(h) \sim |\log h|^{-\frac{1-\alpha}{2}}$ for any \(\alpha > 0\). We will introduce damped quantum maps in Subsection \ref{relatedwork}.

\subsection{Application to Zeros of the Twisted Selberg Zeta Function}
As an application, we consider another important type of non-self-adjoint perturbation of the Laplacian, the twisted Laplacian \( \Delta_{\omega} \), associated with a harmonic 1-form \( \omega \in \mathcal{H}^{1}(M,\mathbb{C}) \) on a compact hyperbolic surface \( M \). The twisted Laplacian \( \Delta_{\omega} \) is defined as follows:
\begin{equation}
    \Delta_{\omega}f := \Delta f - 2\langle \omega, df \rangle_{x} + |\omega|_{x}^{2}f, \quad \text{for} \quad f \in C^{\infty}(M).
\end{equation}
Here, \( |\omega|_{x}^{2} = \langle \omega, \omega \rangle_{x} \)\footnote{Here, $\langle -,- \rangle$ is $\mathbb{C}$-bilinear, i.e., $\langle \lambda \omega, \mu \eta \rangle_{x} = \lambda \mu \langle \omega, \eta \rangle_{x}$ for any $\lambda, \mu \in \mathbb{C}$ and $\omega, \eta \in T^{*}_{x}X$.}. In general, \( \Delta_{\omega} \) is a non-self-adjoint perturbation of the Laplacian and has a discrete spectrum on \( L^{2}(M) \):
\begin{equation}\label{equation}
    \Delta_{\omega}\phi_{j} + \lambda_{j}\phi_{j} = 0, \quad ||\phi_{j}||_{L^2}=1, \quad \text{with} \quad \lambda_{0}<\Re \lambda_{1} \leq \cdots.
\end{equation}
Let the spectral parameter \( r_{j} \) be defined by \( \lambda_{j} = \frac{1}{4} + r_{j}^{2} \), where \( \Re r_{j} \geq 0 \). As a corollary of our main Theorem \ref{mainthm1}, we derive that there exists a constant $c(\omega,M) \in (0,\infty]$\footnote{Here, we identify \(\omega\) as a function \(\omega(x, \xi) := \langle \omega, \xi \rangle_{x}\) on \(T^*M\).} which is defined in \eqref{eq82} such that for any $0<c<c(\omega, M)$ and any $0<\alpha<1$, we have:
\begin{equation}\label{eq59}
\#\left\{n \ \middle| \ 0 \leq \Re r_n \leq \lambda, \ \left|\Im r_n \right|\geq (\log \lambda)^{-\frac{1-\alpha}{2}} \right\}=\mathcal{O}\left(\frac{\lambda^2}{e^{c|\log \lambda|^{\alpha}}(\log \lambda)^{\alpha-1}}\right), \quad \lambda \to \infty.
\end{equation}

Selberg \cite{SA} introduced the famous Selberg zeta function on compact hyperbolic surfaces. M{\"u}ller \cite{MW}, Spilioti \cite{SP}, and Frahm and Spilioti \cite{FS} studied twisted Selberg zeta functions for non-unitary representations. In our setting, we consider the following twisted Selberg zeta function:
$$Z_{\omega}(s):=\prod_{k \geq 0} \prod_{\gamma \in \mathcal{P}(M)}\left(1-e^{\int_{\gamma}\omega}e^{-(s+k) \ell_{\gamma}}\right), \quad \Re s \ \text{large enough}.$$
Here \( \mathcal{P}(M) \) denotes the set of oriented prime geodesics on \( M \). Then \( Z_{\omega}(s) \) has a holomorphic extension to \( \mathbb{C} \). Its set of zeros is the divisor (with multiplicities):
$$\bigcup_j\left\{\frac{1}{2} \pm i r_j\right\} \cup \bigcup_{k=0}^{\infty}(2g-2)(2k+1)\{-k\},$$
where \( \left\{\frac{1}{4}+r_j^2\right\} \) is the spectrum of \( \Delta_{\omega} \), with eigenvalues repeated according to their multiplicity, see Naud and Spilioti \cite[Theorem 2.1]{NS}. We call the $\{-k\}$ trivial zeros and $\left\{\frac{1}{2} \pm i r_j \right\}$ non-trivial zeros of $Z_{\omega}(s)$. Thus, we immediately obtain that most of the non-trivial zeros of \( Z_{\omega}(s) \) are distributed within a logarithmic region, asymptotically approaching \( \Re s = \frac{1}{2} \) as \( \Im s \to \infty \).

\begin{theorem}\label{zerosoftsz}
  Let \( M \) be a compact hyperbolic surface, \( \omega \in \mathcal{H}^{1}(M,\mathbb{C}) \), and \( \{s_{n}\} \) be the non-trivial zeros of the twisted Selberg zeta function \( Z_{\omega}(s) \). We derive that for any $0<c<c(\omega,M)$ and any $0<\alpha<1$, we have:
  $$\#\left\{n \ \middle| \ 0 \leq \Im s_n \leq \lambda, \ \left|\Re s_n - \frac{1}{2} \right|\geq (\log \lambda)^{-\frac{1-\alpha}{2}} \right\}=\mathcal{O}\left(\frac{\lambda^2}{e^{c|\log \lambda|^{\alpha}}(\log \lambda)^{\alpha-1}}\right), \quad \lambda \to \infty.$$
\end{theorem}

\subsection{Related work}\label{relatedwork}
The real and imaginary parts of the complex eigenvalues for \eqref{DWEop} correspond to the frequencies of oscillation and the exponential decay rates of energy, respectively, for the stationary damped wave. Hitrik \cite{HM03} shows that in the case of geometric control, the solution of \eqref{DWEeq} can be expressed as a sum of finite stationary solutions corresponding to underdamped modes and a decaying remainder. Therefore, the spectral distribution is closely related to the structure of the damped wave.

The spectral distribution results in \eqref{eq14}, \eqref{eq30}, and \eqref{eq10} describe upper bounds on the number of eigenvalues with imaginary parts deviating from the average and real parts in the interval \(\left[\frac{1}{2} - Ch, \frac{1}{2} + Ch\right]\). The upper limit of the imaginary parts of the eigenvalues for damped waves in the high-frequency limit, which we call the essential spectral gap, is closely related to the energy decay of solutions to the damped wave equation. For related work, see Lebeau \cite{LB96}, Schenck \cite{SE}, Jin \cite{JL}, and Dyatlov, Jin, and Nonnenmacher \cite{DJN}. On the other hand, Anantharaman \cite{AN2} applied the twisted Selberg trace formula to establish a lower bound on the number of eigenvalues of the twisted Laplacian \(\Delta_{\omega}\) whose imaginary parts deviate from the average on compact arithmetic hyperbolic surfaces. The study of the spectral distribution of both the Laplacian and its non-self-adjoint perturbations in the high-frequency limit on compact Anosov manifolds falls in the area of quantum chaos; see Nonnenmacher \cite{Non11}, Anantharaman \cite{AN3}, and Zelditch \cite{ZS19} for recent surveys.

The damped quantum map is a toy model that was primarily introduced in \cite{KNS08} to mimick the resonance spectra of dielectric microcavities. It consists of an area-preserving smooth diffeomorphism \(\kappa\) acting on the \(2\)-dimensional torus \(\mathbb{T}^2\), along with a damping function \(a(x,\xi) \in C^{\infty}\left(\mathbb{T}^2\right)\), where \(|a| \leq 1\). If we denote by \(U_h(\kappa)\) the unitary propagator obtained from the quantization of \(\kappa\) and by \(\mathrm{Op}_h(a)\) the quantization of the damping function \(a(x,\xi)\), the damped quantum map takes the form \(M_h(a, \kappa):=\mathrm{Op}_h(a) U_h(\kappa)\). The damped quantum map is also a toy model for the damped wave equation on a compact Riemannian manifold. The semiclassical damped propagator \(U_{q}(t):=e^{-itP(z,h)/h}\), where \(P(z,h)\) is the semiclassical damped wave operator \eqref{eq7}, has a form similar to that of the damped quantum map \(M_h(a, \kappa)\); see, for example, \cite[Proposition 6]{Non11}.

Schenck \cite{SE09} studied the spectral distribution of the damped quantum map \(M_h(a, \kappa)\) with an Anosov map \(\kappa\). He established the Weyl law, the width of the spectral distribution, and a spectral deviation theorem for the damped quantum map \(M_{h}(a,\kappa)\) in \cite[Theorem 1.1-1.4]{SE09}. These results all have corresponding analogues in the context of damped wave equations on a compact Anosov manifold; see \eqref{eq24}, \eqref{eq14}, \eqref{eq30}, and Theorem \ref{mainthm1}. Numerical results for damped quantum maps are discussed in Nonnenmacher and Schenck \cite{NS08}.

In hyperbolic geometry and dynamical systems, the twisted Laplacian is used to study the number of closed geodesics in a homology class on a compact hyperbolic surface; see Katsuda and Sunada \cite{KS}, Phillips and Sarnak \cite{PS}, Lalley \cite{LS1,LS2}, Babillot and Ledrappier \cite{BL}, and Anantharaman \cite{AN1}. It is also considered as the Bochner Laplacian on a flat line bundle, which is constructed from a 1-dimensional representation of the fundamental group (see \cite{Kob}). Naud and Spilioti \cite{NS} proved the Weyl law for some higher-dimensional non-unitary twisted Bochner Laplacians using the method of thermodynamic formalism.

A monograph by Sjöstrand \cite{SJ2} systematically studies the spectral distribution of non-self-adjoint perturbations of second-order elliptic operators; see also Gohberg and Krein \cite{GK69}. 

\subsection*{Organization of the paper}
The paper is organized as follows: In Section \ref{sec2}, we introduce the MDP Theorem \ref{MDPA} for a compact Anosov manifold and the semiclassical pseudodifferential calculus. In Section \ref{sec3}, we construct the average of the operator $\mathcal{P}_{T}(z)$ and its perturbation $\widetilde{\mathcal{P}}_{T}(z)$. In Section \ref{sec4}, we control the trace norm of the perturbation $\widetilde{\mathcal{P}}_{T}(z)-\mathcal{P}_{T}(z)$ and prove the invertibility of $\widetilde{\mathcal{P}}_{T}(z)-z$. In Section \ref{sec5}, we prove Theorems \ref{mainthm1}, \ref{concentrationforDWE} and \ref{zerosoftsz}. In Appendix \ref{appendix}, we give a proof of the MDP Theorem \ref{MDPA}.

\subsection*{Notation}
In the following, when the dependence of a constant on parameters is crucial, we indicate the parameters on which they depend. Big $\mathcal{O}$ notation means that given two functions $f(x)$ and $g(x)$, we say $f(x)=\mathcal{O}(g(x))$ as $x \to A$, ($A \in [-\infty,\infty]$) if there exist positive constants $C$ such that $|f(x)| \leq C|g(x)|$ for all $x$ which are sufficiently closed to $A$.  If the constant depends on some parameter $\alpha$, we write $f=\mathcal{O}_{\alpha}(g)$.

\subsection*{Acknowledgements}
The author would like to thank Long Jin for his encouragement and numerous discussions, and St\'ephane Nonnenmacher for suggesting the problem of the width of concentration for the spectrum of damped waves as well as for providing many useful suggestions, including the improvement of width \(w(h)\) to \(|\log h|^{-\frac{1-\alpha}{2}}\) for any \(\alpha > 0\). I also thank Jinxin Xue and Weili Zhang for several discussions on moderate deviations and analytic perturbations of operators. Yulin Gong is supported by the National Key R \& D Program of China 2022YFA100740.

\section{Preliminary}\label{sec2}
\subsection{Anosov Geodesic Flow on Compact Manifolds}
In this subsection, we review the basis of Anosov geodesic flow, for more details, see \cite{KH, BS02}. Furthermore, we introduce the MDP Theorem \ref{MDPA}, which provides the subexponential decay of the measure of moderate deviation sets.
\begin{definition}\label{anosov}
The geodesic flow $\varphi^{t}$ on the sphere bundle $S^*M$ of a $d$-dimensional compact smooth manifold $M$ is called an Anosov flow if, at each point $\rho \in S^*M$, there exists a splitting of the tangent space:
$$
T_\rho S^* M=\mathbb{R} X(\rho) \oplus E^{+}(\rho) \oplus E^{-}(\rho),
$$
where $X(\rho)$ is the Hamiltonian vector field generated by $\varphi^{t}$, and $E^{ \pm}(\rho)$ are the unstable and stable subspaces at the point $\rho$. Both subspaces have dimension $d-1$. The families $\left\{E^{ \pm}(\rho), \rho \in S^*M\right\}$ form the unstable/stable distributions, which are invariant with respect to the flow, Hölder continuous, and characterized by the following property: there exist constants $C, \lambda>0$ such that
\begin{equation}\label{eq26}
\forall \rho \in S^*M, \ \forall v \in E^{\mp}(\rho), \ \forall t>0, \quad \left|d \varphi_\rho^{ \pm t} (v)\right| \leq C e^{-\lambda t}|v| .
\end{equation}
If the geodesic flow $\varphi^{t}$ of $M$ is Anosov, then $M$ is called an Anosov manifold. One broad class of Anosov manifolds is that of negatively curved manifolds. 
\end{definition}

The chaotic properties of Anosov flow are widely studied. In particular, Hopf proved that the geodesic flow of a compact hyperbolic surface is ergodic with respect to the Liouville measure on the sphere bundle. Anosov and Sinai applied Hopf's argument to prove that the Anosov geodesic flow on a compact manifold is ergodic with respect to the Liouville measure on the sphere bundle. These results can be found in \cite{KH, BS02}.

Furthermore, Dolgopyat \cite{DD98} proved that the geodesic flow on a compact Anosov surface is exponentially mixing. Liverani \cite{LC04} extended Dolgopyat's theorem to general contact Anosov flows. We note that the geodesic flow is a special case of a contact flow on the cosphere bundle. These proofs are based on the study of the transfer operator (see \eqref{transfer}), and research on the transfer operator can be found in, e.g., \cite{DD98, LC04, TM10, FS11, DZ16}.  

In Appendix \ref{appendix}, we introduce Tsujii's theorem \cite{TM10}, which establishes the quasi-compactness of the transfer operator (see Theorem \ref{transfer}). The quasi-compactness of the transfer operator plays a crucial role in the proof of the MDP (Theorem \ref{MDPA}).  

In probability theory, the large deviation principle (LDP) and central limit theorems (CLT) have been widely studied (see, e.g., \cite{DR19}). The MDP is an intermediate regime between the LDP and the CLT (see, e.g., \cite{WL95, DC98, DA09, DS23}).  

In hyperbolic dynamical systems, the LDP has been established by Kifer \cite{KY} and Young \cite{LSY}, while the CLT for Anosov geodesic flows was proved by Ratner \cite{RM73}. The MDP has also been studied in some non-uniformly hyperbolic systems, including Anosov diffeomorphisms (see, e.g., Rey-Bellet and Young \cite{RY08}).  

Now, we define an equivalence relation on \( C^{\infty}(S^*M, \mathbb{R}) \).

\begin{definition}\label{cohomologous}
For any $f,g \in C^{\infty}(S^*M,\mathbb{R})$, we call $f$ is cohomologous to $g$ if there exists $h \in C^{\infty}(S^*M,\mathbb{R})$ such that 
\begin{equation}\label{cohomoeq}
f(\rho)=g(\rho)+Xh(\rho), \quad \rho \in S^*M.
\end{equation}
\end{definition}
Ratner's theorem \cite[Theorem 3.1]{RM73} implies that if $M$ is a compact Anosov manifold, for any $q\in C^{\infty}(S^*M)$, there exists a constant $\sigma_{q}\geq 0$ such that
\begin{equation}\label{limitofvariance}
\lim\limits_{t \to \infty}\frac{1}{t}\int_{S^*M}\left|q \circ \varphi^{t}(\rho)-\overline{q}\right|^2d\mu_{L}(\rho)=\sigma_{q}^2.
\end{equation}
Furthermore, $\sigma_{q}>0$ if and only if $q$ is not cohomologous to a constant, i.e. there does not exist a smooth function such that $q=\overline{q}+Xh$. Now we introduce the MDP for Anosov geodesic flows on a compact manifold:
  \begin{theorem}[Moderate deviations]\label{MDPA}
  Let $M$ be a compact Anosov manifold, $a(t)$ be a family of positive real numbers such that $\lim\limits_{t \rightarrow \infty}a(t)t^{-\frac{1}{2}}=0$ and $\lim\limits_{t \rightarrow \infty} a(t)=\infty$. For any real-valued smooth function $q \in C^{\infty}(S^*M, \mathbb{R})$ which is not cohomologous to a constant, and any closed interval $I \subset \mathbb{R}$, we have
  \begin{equation}\label{moderatedeviationprinciples}
  \lim _{t \rightarrow \infty} \frac{a(t)^2}{ t} \log \mu_{L}\left\{\rho \in S^*M \ \Bigg| \ a(t)\left(\frac{1}{t}\int_{0}^{t}q \circ \varphi^{s}ds-\overline{q}\right) \in I \right\}=-\inf _{s \in I} \frac{s^2}{2 \sigma^{2}_{q}}
  \end{equation}
  \end{theorem}
  If $q$ is cohomologous to a constant, i.e. there is a smooth function $h$ such that $q=\overline{q}+Xh$, we have:
  \begin{equation}\label{cohomo}
    \left|\frac{1}{t}\int_{0}^{t}q \circ \varphi^{s}(\rho)ds-\overline{q}\right|=\left|\frac{1}{t}\int_{0}^{t}Xh \circ \varphi^{s}(\rho)ds\right|\leq \frac{2\|h\|_{\infty}}{t}.
    \end{equation}

  In this paper, we consider the following symmetric average of $q$:
  \begin{equation}\label{symmetricaverage}
    \langle q \rangle_{T}(\rho)=\frac{1}{T}\int_{-T/2}^{T/2}q \circ \varphi^{t}(\rho)dt.
  \end{equation}
  Combine \eqref{moderatedeviationprinciples} and \eqref{cohomo}, we derive that:
  \begin{corollary}\label{coroMDPA}
  For any smooth function $q \in C^{\infty}(S^*M, \mathbb{R})$, $\varepsilon>0$, $0<c<\frac{\varepsilon^2}{2\sigma_{q}^2}$, and $a(T)$ satisfying the condition in Theorem \ref{MDPA}, and sufficiently large $T$, we have
  \begin{equation}\label{moderatedeviatearea}
    \mu_{L}\left\{\rho \in S^*M \ \Bigg| \ \langle q\rangle_{T}(\rho)-\overline{q}\geq \frac{\varepsilon}{a(T)} \right\}\leq e^{-cTa(T)^{-2}}.
    \end{equation}
  \end{corollary}
  We notice that $a(t)=1$ corresponds to the LDP and $a(t)=t^{\frac{1}{2}}$ corresponds to the CLT. We will give a complete proof of Theorem \ref{MDPA} in Appendix \ref{appendix}, our method comes from Rey-Bellet and Young \cite{RY08}, Wu \cite{WL95}, and Dolgopyat and Sarig \cite{DS23}.
  \begin{remark}
    Waddington \cite{WS96} proved precise asymptotics for large deviations in \(C^2\) transitive Anosov flows. We believe that the same method can be applied to moderate deviations for \(C^2\) transitive Anosov flows to derive precise asymptotics.
  \end{remark}
\subsection{Semiclassical Microlocal Analysis}
In this subsection, we briefly review semiclassical pseudodifferential operators and refer the reader to \cite{ZM, DS, DZ19} for an introduction. In our proof, we need to consider the evolution of the symbol up to time \(t = \mu|\log h|\) for some constant \(\mu > 0\). Therefore, we consider the following exotic symbol classes: for any \(k \in \mathbb{R}\), the exotic symbol classes \(S_\delta^k\left(T^* M\right)\), with \(\delta \in \left[0, \frac{1}{2}\right)\), on a Riemannian manifold \(M\) consist of symbols satisfying the following derivative bounds:
\begin{equation}\label{eq53}
  \sup _{h \in\left(0, 1\right]} h^{\delta(|\alpha|+|\beta|)} \sup _{\substack{x \in K\Subset M \\ \xi \in T_x^* M}}\langle\xi\rangle^{|\beta|-k}\left|\partial_x^\alpha \partial_{\xi}^\beta a(x, \xi ; h)\right|<\infty.
\end{equation}
That is, the derivative $\partial_x^\alpha \partial_{\xi}^\beta a$ is allowed to grow like $h^{-\delta(|\alpha|+|\beta|)}\langle\xi\rangle^{k-|\beta|}$ uniformly on each compact subset $K \Subset M$. We also define $$S^{-\infty}_{\delta}(T^*M)=\bigcap_{k\in \mathbb{R}}S^{k}_{\delta}(T^*M).$$
Now we define a (non-canonical) quantization procedure on a compact manifold $M$:
\begin{definition}\label{defofquanandprin}
    Let $\left(\varphi_j \colon U_j \mapsto V_j \subset \mathbb{R}^d, \chi_j\right)$ be cutoff charts of $M$, i.e., $\chi_j \in C_{0}^{\infty}(U_j)$, and $\chi_j^{\prime}$ functions satisfying $\chi_j^{\prime}=1$ near $\mathrm{supp}(\chi_j)$. Then, for any $a \in S_\delta^k\left(T^* M\right), \delta \in\left[0, \frac{1}{2}\right)$, the operator $\mathrm{Op}_{h}(a)\colon C^{\infty}(M)\mapsto C^{\infty}(M)$ is defined by: 
$$
\mathrm{Op}_h(a):=\sum_{j} \chi_j^{\prime} \varphi_j^* \operatorname{Op}_h^{\mathrm{KN}}\left(\left(\chi_j a\right) \circ \widetilde{\varphi}_j^{-1}\right)\left(\varphi_j^{-1}\right)^* \chi_j^{\prime},
$$
where for $b \in S^{k}_{\delta}(T^*\mathbb{R}^d)$, the Kohn-Nirenberg quantization of the symbol $b$ is the operator $\mathrm{Op}_{h}^{\mathrm{KN}}(b)$ which is defined as 
$$\mathrm{Op}_{h}^{\mathrm{KN}}(b)u(x):=\left(\frac{1}{2\pi h}\right)^{d}\int_{\mathbb{R}^{2d}}e^{\frac{i(x-y)\cdot \xi}{h}}b(x,\xi)u(y)d\xi dy, \quad u\in C^{\infty}(\mathbb{R}^{d}).$$
Here, $\widetilde{\varphi}_j \colon T^*U_j \mapsto T^*V_j\subset T^*\mathbb{R}^d$ is defined by $\widetilde{\varphi}_{j}(x,\xi)=(\varphi_{j}(x), ((d\varphi_{j})_{x}^{t})^{-1}(\xi))$, which is the symplectic lifting of $\varphi_{j}$. For $\delta \in \left[0, \frac{1}{2}\right)$, the pseudodifferential operator classes $\Psi_{\delta}^{k}(M)$ consist of $\mathrm{Op}_{h}(a)+R$, where $a \in S^{k}_{\delta}(T^*M)$, $R=h^{\infty}\Psi^{-\infty}(M)$, and $\Psi^{-\infty}(M)$ denotes operators with smooth integral kernels. When $\delta=0$, we omit the subscript $S^{k}(T^*M):=S^{k}_{0}(T^*M)$ and $\Psi^{k}(M):=\Psi_{0}^{k}(M)$. Conversely, for any $A \in \Psi_{\delta}^{k}(M)$, $\sigma_{h}(A) \in S_{\delta}^{k}(T^*M)$ is called the principal symbol of $A$, if $$A-\mathrm{Op}_{h}\left(\sigma_{h}(A)\right) \in h^{1-2\delta}\Psi_{\delta}^{k-1}(M).$$ This implies that for any $a \in S_{\delta}^{k}(T^*M)$, $a-\sigma_{h}\left(\mathrm{Op}_{h}(a)\right)\in h^{1-2\delta}S_{\delta}^{k-1}(T^*M)$.
\end{definition}
We list some properties of the quantization $\mathrm{Op}_{h}$ and principal symbol map:
\begin{proposition}\label{prop1}
    For any $m,n \in \mathbb{R}$, let $a \in S_\delta^m(T^*M)$ and $b \in S_{\delta^{\prime}}^n(T^*M)$ with $0 \leq \delta^{\prime} \leq \delta<1 / 2$. Then,
    \begin{itemize}
        \item $\mathrm{Op}_{h}(a)\mathrm{Op}_{h}(b)-\mathrm{Op}_{h}(ab) \in h^{1-\delta-\delta^{\prime}} \Psi_{\delta}^{m+n-1}(M)$.
        \item $\left[\mathrm{Op}_{h}(a), \mathrm{Op}_{h}(b)\right]-\frac{h}{i} \mathrm{Op}_{h}(\{a, b\}) \in h^{2\left(1-\delta-\delta^{\prime}\right)} \Psi_{\delta}^{m+n-2}(M)$.
        \item $\mathrm{Op}_{h}(a)^*-\mathrm{Op}_{h}(\overline{a}) \in h^{1-2\delta}\Psi_{\delta}^{m-1}(M)$.
    \end{itemize}
\end{proposition}
The sharp Gårding inequality shows that:
\begin{theorem}\label{Garding}
    If $a \in S^{0}_{\delta}(T^*M)$ with $a \geq 0$, then there exists $C>0$ such that for any $u \in L^2(M)$:
    $$\Re \langle \mathrm{Op}_{h}(a)u,u\rangle \geq -Ch^{1-2\delta}\|u\|_{L^2}^2.$$
    As a corollary, we have for any $a \in S^{0}_{\delta}(T^*M)$:
    $$\left\|\mathrm{Op}_{h}(a)\right\| \leq \sup_{T^*M} |a(x,\xi)|+\mathcal{O}(h^{1-2\delta}).$$
\end{theorem}
For our proof, we also need the local form of sharp Gårding's inequality, similar to the local form of the Calderón-Vaillancourt inequality (see \cite{AN2}):
\begin{theorem}\label{localCVineq}
    Let $M$ be a $d$-dimensional compact Riemannian manifold. Take $a \in S_\delta^2(T^*M)$ with $0 \leq \delta<1/2$. Let $I$ be an open interval with compact closure $\overline{I}\subset \mathbb{R}^{+}$ and let $\lambda$ belong to $I$. Then, for all $u \in H^2(M)$:
    \begin{equation}
      \begin{gathered}
    \|\mathrm{Op}_{h}(a) u\|_{L^2} \leq \left(\sup\limits_{p^{-1}(I)} |a|+\mathcal{O}(h^{1-2\delta})\right)\|u\|_{L^2}+\mathcal{O}(1)\|(P-\lambda) u\|_{L^2}\\
    \Re \langle \mathrm{Op}_{h}(a) u, u\rangle \geq \left(\inf\limits_{p^{-1}(I)} |a|-\mathcal{O}(h^{1-2\delta})\right)\|u\|_{L^2}^2-\mathcal{O}(1)\|(P-\lambda) u\|_{L^2}\|u\|_{L^2}.
      \end{gathered}
    \end{equation}
    Here, $P\in \Psi^2(M)$ with $\sigma_{h}(P)(x,\xi)=p(x,\xi)=\frac{1}{2}|\xi|_{x}^2$.
\end{theorem}

In what follows, we always assume that $M$ is a $d$-dimensional compact Anosov manifold. Let the Hamiltonian be $p(x,\xi)=\frac{1}{2}|\xi|_{x}^2$ on $T^*M$, and note that the geodesic flow $\varphi^{t}$ is generated by the Hamiltonian vector field $X$ associated with $p$. Note that $P \in \Psi^2(M)$ with $\sigma_{h}(P)=p$. We fix an arbitrary $C>0$ as the width of the spectral window in \eqref{eq9}.
\section{Averaging and Perturbations}\label{sec3}
Since the conjugation of the operator $P(z,h)$ does not change the spectrum of $P(z,h)$, we follow the construction of the conjugate operator as outlined in Sjöstrand \cite{SJ1}. Recall that the operator $A_T$ is defined as $A_{T}=\mathrm{Op}_{h}\left(e^{g_T}\right)$ (see Definition \ref{defofquanandprin}). The function $g_T$ is constructed as follows. First, we define
\begin{equation}
\widetilde{g}_T := \frac{1}{2} \int_0^{T / 2}\left(\frac{2 s}{T}-1\right) q \circ \varphi^s \, d s + \frac{1}{2} \int_{-T / 2}^0\left(\frac{2 s}{T}+1\right) q \circ \varphi^s \, d s.
\end{equation}
The function $\widetilde{g}_T$ satisfies $\left\{p, \widetilde{g}_T\right\}=q-\langle q\rangle_T$ on $T^*M\setminus 0$. We fix a small $0<\epsilon_{0}<1/8$, and choose $\chi_{0}(t) \in C_{0}^{\infty}(\mathbb{R})$ such that $\chi_{0}=1$ on $\left[\frac{1}{2}-\epsilon_{0},\frac{1}{2}+\epsilon_0\right]$ and $\mathrm{supp}(\chi_{0})\subset \left(\frac{1}{2}-2\epsilon_{0},\frac{1}{2}+2\epsilon_0\right)$. Then we define 
\begin{equation}\label{eq11}
  g_T(x,\xi) = \widetilde{g}_{T}(x,\xi) \chi_{0}(p(x,\xi)).
\end{equation}
Thus, for any fixed $T>0$, we have $e^{g_{T}}(x,\xi) \in S^{0}(T^*M)$, and $e^{g_{T}}$ is strictly positive on $T^*M$. By the Calderón-Vaillancourt inequality and Gårding's inequality, $A_{T}=\mathrm{Op}_{h}\left(e^{g_T}\right)$ is bounded and invertible on $L^2(M)$. In Anantharaman \cite{AN2}, she considers the long time $T$ such that $T$ is close to the Ehrenfest time on $M$. We define the maximal expansion rate of geodesic flow on the energy shell and cosphere bundle respectively as follows,
\begin{equation}\label{maximalexpansion}
  \begin{gathered}
\Lambda:=\lim _{t \rightarrow \infty} \sup _{\left|p(\rho)-\frac{1}{2}\right|\leq 2\epsilon_{0}} \frac{1}{t} \log \left\|d\varphi_{\rho}^{t}\right\| \\
\Lambda_{0}:=\lim _{t \rightarrow \infty} \sup _{\rho \in S^*M} \frac{1}{t} \log \left\|d\varphi_{\rho}^{t}\right\|.
  \end{gathered}
\end{equation}
Now we fix $\epsilon>0$, and the Ehrenfest time is defined as follows:
\begin{equation}
    T = \frac{(1-4\epsilon) |\log h|}{\Lambda}.
\end{equation}
\begin{proposition}[Anantharaman \cite{AN2}]
There exists a bounded and invertible operator on $L^2(M)$, $A_T \in \Psi_{\frac{1}{2}-\epsilon}^{0}(M)$, such that for any $\left|z-\frac{1}{2}\right|=\mathcal{O}(h)$, we have
\begin{equation}
  \mathcal{P}_T(z,h) := A_T^{-1}(P + i h Q(z)) A_T = P + i h Q_{T} + h \widetilde{R}_T(z),
\end{equation}
where $Q_{T}=\mathrm{Op}_{h}\left(q^T\right) \in \Psi_{\frac{1}{2}-\epsilon}^{1}(M)$, $q^T \in S^{1}_{\frac{1}{2}-\epsilon}(T^*M)$, $q^{T} = \langle q\rangle_T$ in $p^{-1}\left(\left[\frac{1}{2}-\epsilon_{0},\frac{1}{2}+\epsilon_0\right]\right)$, and $\widetilde{R}_T(z)\in h^{\epsilon} \Psi_{\frac{1}{2}-\epsilon}^{1}(M)$ depends holomorphically on $z$. 
\end{proposition}
For any $\left|z-\frac{1}{2}\right|=\mathcal{O}(h)$, by Theorem \ref{localCVineq}, for any $u \in L^2(M)$, we have:
\begin{equation}\label{eq31}
  \|\widetilde{R}_T(z) u\|_{L^2} = \mathcal{O}\left(h^{\epsilon}\right)\left(\|u\|_{L^2} + \|\left(2P-1\right)\|_{L^2}\right).
\end{equation}
In the following context, we write $\mathcal{P}_{T} = \mathcal{P}_T(z,h)$ when it cannot cause confusion. Next, we construct a perturbation $\widetilde{\mathcal{P}}_{T}$ of $\mathcal{P}_{T}$ such that there is good control over $(\widetilde{\mathcal{P}}_{T} - z)^{-1}$ and the trace class norm $\left\|\widetilde{\mathcal{P}}_{T} - \mathcal{P}_{T}\right\|_{\mathrm{tr}}$. We want to change the poles of $(z-\mathcal{P}_T(z,h))^{-1}$ to the zeros of the determinant of the trace class operator in a given open set $\Omega_{h}$. Thus, we can apply Jensen's formula to give an upper bound on the number of poles in $\Omega_{h}$. Our construction also follows Sjöstrand \cite{SJ1} and Anantharaman \cite{AN2}; however, our domain $\Omega_{h}$ depends on $h$, so we must deal with the $h$ dependence carefully. \\
\indent We choose a width $w(h)$ satisfying $1/w(h) = o\left(|\log h|^{1/2}\right)$ as $h \to 0$. Now we construct a smooth function $\chi(E) \in C^{\infty}(\mathbb{R}, \mathbb{R})$ satisfying:
\begin{itemize}
  \item $\chi(E) = E$ for $E \leq 0$, and $\chi(E) = \frac{1}{2}$ for $E \geq 1$;\\
  \item $\chi(E) \leq \frac{1}{2}$ for all $E \in \mathbb{R}$;\\
  \item $0 \leq \chi^{\prime}(E) \leq 1$ and $\mathrm{supp}(\chi^{\prime}(E)) \subset [0,1]$.
\end{itemize}
Now we choose fix an arbitrary $0<\kappa<1$, we define an $h$-dependent function $\chi_{h}(E) \in C^{\infty}(\mathbb{R}, \mathbb{R})$: 
\begin{equation}\label{eq20}
  \chi_{h}(E) = \frac{(1-\kappa)w(h)}{2}\chi\left(\frac{E - \overline{q} - \kappa w(h)}{\frac{1-\kappa}{2}w(h)}\right) + \overline{q} + \kappa w(h).
\end{equation}
It satisfies:
\begin{itemize}
  \item $\chi_{h}(E) = E$ for $E \leq \overline{q} + \kappa w(h)$; \\ 
  \item $\chi_{h}(E) \leq \overline{q} + \frac{1+3\kappa}{4}w(h)$ for all $E \in \mathbb{R}$;\\
  \item $0 \leq \chi^{\prime}_{h}(E) \leq 1$ and $\mathrm{supp}(\chi^{\prime}_{h}(E)) \subset \left[\overline{q} + \kappa w(h), \overline{q} + \frac{1+\kappa }{2}w(h)\right]$.
\end{itemize}
Furthermore, for any $k \in \mathbb{N}^{+}$ and any $E \in \mathbb{R}$, 
$$\frac{\partial^{k} \chi_{h}(E)}{\partial E^{k}} = \mathcal{O}\left(\frac{1}{w(h)^{k-1}}\right) = \mathcal{O}\left(|\log h|^{\frac{k-1}{2}}\right).$$
Let $\delta = \frac{1-\epsilon}{2} > \frac{1}{2} - \epsilon$. By Faà di Bruno's formula, it follows that $\chi_{h}(q^T) \in S_{\delta}^{1}(T^*M)$. Let $$\widehat{q}^{T}(\rho) = \left(q^{T}(\rho) - \chi_{h}(q^T)\right)\chi_{0}(\rho), \quad \rho \in T^*M,$$ 
see Figure \ref{differenttruncate}. We have $0 \leq \widehat{q}^{T} \in S^{-\infty}_{\delta}(T^*M)$. Then we define
\begin{equation}
  \mathcal{Q}_{T} := \frac{1}{2}\left(\mathrm{Op}_{h}\left(\widehat{q}^{T}(x,\xi)\right) + \left(\mathrm{Op}_{h}\left(\widehat{q}^{T}(x,\xi)\right)\right)^{*}\right) \in \Psi_{\delta}^{-\infty}(M).
\end{equation}
It follows that $\mathcal{Q}_{T}$ is a self-adjoint operator on $L^2(M)$. 

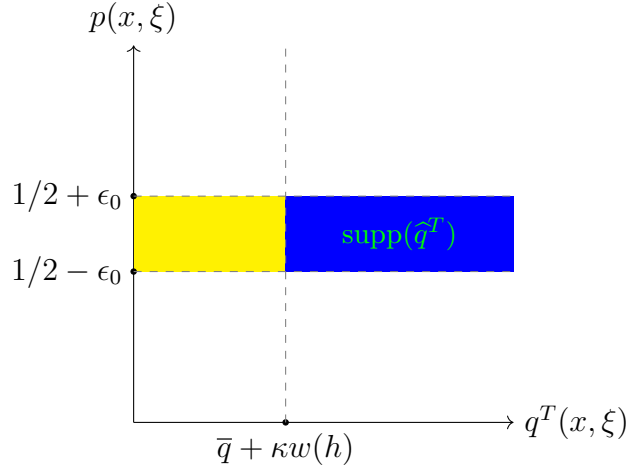
\begin{figure}
  \centering
\begin{tikzpicture}[xscale=5, yscale=5]
  \draw[->] (0,0) -- (1,0) node[right] {$q^T(x,\xi)$};
  \draw[->] (0,0) -- (0,1) node[above] {$p(x,\xi)$};
  
  \fill[yellow] (0,0.4) -- (0,0.6) -- (0.4,0.6) -- (0.4,0.4) -- cycle;

  \fill[blue] (0.4,0.4) -- (0.4,0.6) -- (1,0.6) -- (1,0.4) -- cycle;
  \filldraw (0.4,0) circle (0.2pt) node[below]{$\overline{q}+\kappa w(h)$};
  \filldraw (0,0.4) circle (0.2pt) node[left]{$1/2-\epsilon_{0}$};
  \filldraw (0,0.6) circle (0.2pt) node[left]{$1/2+\epsilon_{0}$};
  \draw[dashed, gray] (0.4,0) -- (0.4,1);
  \draw[dashed, gray] (0,0.4) -- (1,0.4);
  \draw[dashed, gray] (0,0.6) -- (1,0.6);
  \node[font=\small, text=green] at (0.7, 0.5) {\(\mathrm{supp}(\widehat{q}^{T})\)};
\end{tikzpicture}
\caption{The horizontal axis represents \(q^{T}(x,\xi)\), and the vertical axis represents \(p(x,\xi)\). The colored region represents the energy shell \(p^{-1}\left(\left[\frac{1}{2}-\epsilon_{0},\frac{1}{2}+\epsilon_{0}\right]\right)\), where \(q^{T} = \langle q \rangle_{T}\) in this region. The yellow region represents the part where \(\langle q \rangle_{T} \leq \overline{q} + \kappa w(h)\), and in this region, \(q^{T} = \chi_{h}(q^{T})\). The blue region represents the part where \(\langle q \rangle_{T} \geq \overline{q} + \kappa w(h)\), and this region contains \(\mathrm{supp}\left(\widehat{q}^{T}\right)\).}
\label{differenttruncate}
\end{figure}

Let $0 \leq f \in \mathcal{S}(\mathbb{R})$ with $f \geq 1$ on $[-10C, 10C]$, and $\widehat{f} \in C_0^{\infty}(\mathbb{R})$, where $\widehat{f}(t) = \int_{\mathbb{R}} e^{-i t E} f(E) \, d E$ is the Fourier transform. Define
\begin{equation}
\widetilde{\mathcal{P}}_{T}(z) = P + i h \widehat{Q}_T + h \widetilde{R}_T(z),
\end{equation}
with
\begin{equation}
\widehat{Q}_T = Q_T - f\left(\frac{2 P - 1}{h}\right)\mathcal{Q}_{T} f\left(\frac{2 P - 1}{h}\right).
\end{equation}
Similar perturbations based on $f\left(\frac{2 P - 1}{h}\right)$ are used by Bony \cite{JFB}. We note that
$$f\left(\frac{2 P - 1}{h}\right) = \frac{1}{2\pi} \int_{\mathbb{R}} e^{it\frac{2P-1}{h}} \widehat{f}(t) \, dt,$$
so $f\left(\frac{2 P - 1}{h}\right)$ is a bounded and trace class operator on $L^2(M)$. Furthermore, we have
\begin{equation}\label{eq58}
  \left\|f\left(\frac{2P-1}{h}\right)\right\| = \mathcal{O}(1), \quad \left\|f\left(\frac{2P-1}{h}\right)\right\|_{\mathrm{tr}} = \mathrm{Tr}\left(f\left(\frac{2P-1}{h}\right)\right) = \mathcal{O}\left(h^{1-d}\right).
\end{equation}

\section{Control the Perturbation}\label{sec4}
In this section, we will prove the control of the trace norm of the perturbation $\widetilde{\mathcal{P}}_{T}-\mathcal{P}_{T}$ and the operator norm of $\left(z-\widetilde{\mathcal{P}}_{T}(z,h)\right)^{-1}$. The following theorem is proved in \cite{SJ1} for fixed $T$ and in \cite{AN2} for $T=\frac{(1-4 \epsilon)|\log h|}{\Lambda}$.

\begin{proposition}[Anantharaman \cite{AN2}]\label{tracethm}
There exist constants $C_{d}>0$ and $0<h_{0}\leq 1$ such that for $0<h\leq h_{0}$, we have:
\begin{equation}\label{eq21}
\begin{aligned}
\left\| \widetilde{\mathcal{P}}_{T}(z,h)- \mathcal{P}_{T}(z,h) \right\|_{\mathrm{tr}} \leq C_d h^{2-d}\left(\|f\|_{L^2(\mathbb{R})}^2 \int_{S^*M}\left(q^T-\chi_{h}\left(q^{T}\right)\right) d\mu_{L}(\rho)+\mathcal{O}(h^{1-2\delta})\right).
\end{aligned}
\end{equation}
\end{proposition}

\begin{proof}
We note that
$$\widetilde{\mathcal{P}}_{T}(z,h)- \mathcal{P}_{T}(z,h)=f\left(\frac{2 P - 1}{h}\right)\mathcal{Q}_{T} f\left(\frac{2 P - 1}{h}\right).$$
By sharp Gårding's inequality \ref{Garding}, there exists $C^{\prime}>0$ such that $$ \mathcal{Q}_{T}+C^{\prime}h^{1-2\delta}\cdot \mathrm{Id}\geq 0.$$ Hence, 
\begin{equation}\label{eq42}
\begin{aligned}
& \left\|f\left(\frac{2P-1}{h}\right)\mathcal{Q}_{T} f\left(\frac{2P-1}{h}\right)\right\|_{\mathrm{tr}} \\ 
& \leq \left\|f\left(\frac{2P-1}{h}\right)(\mathcal{Q}_{T}+C^{\prime} h^{1-2\delta}) f\left(\frac{2P-1}{h}\right)\right\|_{\mathrm{tr}} + C^{\prime} h^{1-2\delta}\left\|f\left(\frac{2P-1}{h}\right)^2\right\|_{\mathrm{tr}} \\ 
& \leq \operatorname{Tr}\left(f\left(\frac{2P-1}{h}\right)(\mathcal{Q}_{T}+C^{\prime} h^{1-2\delta}) f\left(\frac{2P-1}{h}\right)\right)+\mathcal{O}\left(h^{2-d-2\delta}\right) \\ 
& \leq \operatorname{Tr}\left(f\left(\frac{2P-1}{h}\right)\mathcal{Q}_{T} f\left(\frac{2P-1}{h}\right)\right)+\mathcal{O}\left(h^{2-d-2\delta}\right).
\end{aligned}
\end{equation}
Here,
\begin{equation}\label{eq43}
\begin{aligned}
\operatorname{Tr}\left(f\left(\frac{2P-1}{h}\right)\mathcal{Q}_{T} f\left(\frac{2P-1}{h}\right)\right)=\operatorname{Tr}\left(f\left(\frac{2P-1}{h}\right)^2\mathcal{Q}_{T}\right)=\operatorname{Tr} \frac{1}{2 \pi} \int \widehat{f^2}(t) e^{i t \frac{2P-1}{h}}\mathcal{Q}_{T}d t.
\end{aligned}
\end{equation}
Similar to \eqref{eq58}, applying the stationary phase method in the time variables, we obtain the following expansion; see \cite[(5.1)]{AN2}:
\begin{equation}\label{eq47}
\begin{aligned}
\operatorname{Tr} \frac{1}{2 \pi} \int \widehat{f^2}(t) e^{i t \frac{2 P-1}{h}}\mathcal{Q}_{T} d t = C_d h^{1-d}\left[\widehat{f^2}(0) \int_{p^{-1}(1 / 2)}\widehat{q}^{T}(\rho) d\mu_{L}(\rho)+\mathcal{O}\left(h^{(1-2 \delta)}\right)\right],
\end{aligned}
\end{equation}
where $C_d>0$ only depends on the dimension $d$ of $M$. Thus, \eqref{eq43} is equal to
\begin{equation}\label{eq45}
C_d h^{1-d}\left(\int_{S^*M}(q^{T}-\chi_{h}\left(q^{T}\right))(\rho) d\mu_{L}(\rho) \widehat{f^2}(0)+\mathcal{O}(h^{1-2\delta})\right), \quad h \rightarrow 0.
\end{equation}
Furthermore, $\widehat{f^2}(0)=\|f\|_{L^2(\mathbb{R})}^2$, so combining this with \eqref{eq43}, we get
\begin{equation}\label{eq46}
\begin{aligned}
\left\| f\left(\frac{2P-1}{h}\right)\mathcal{Q}_{T} f\left(\frac{2P-1}{h}\right) \right\|_{\mathrm{tr}}
\leq C_{d} h^{1-d} \left(\int_{S^*M}(q^{T}-\chi_{h}\left(q^{T}\right)) d\mu_{L}(\rho)\|f\|_{L^2(\mathbb{R})}^2+\mathcal{O}\left(h^{1-2\delta}\right)\right).
\end{aligned}
\end{equation}
Equations \eqref{eq42} and \eqref{eq46} provide estimates for the operator and trace norms of
\begin{equation}\label{eq54}
\|\widetilde{\mathcal{P}}_{T}-\mathcal{P}_{T}\|_{\mathrm{tr}}\leq C_d h^{2-d}\left(\|f\|_{L^2(\mathbb{R})}^2 \int_{S^*M}\left(q^T-\chi_{h}\left(q^{T}\right)\right) d\mu_{L}(\rho)+\mathcal{O}(h^{1-2\delta})\right).
\end{equation}
\end{proof}

For proof of Theorem \ref{mainthm1}, let $$\|q\|_{\infty}:=\sup\limits_{\rho \in p^{-1}\left(\left[\frac{1}{2}-\epsilon_{0},\frac{1}{2}+\epsilon_{0}\right]\right)}|q(\rho)|.$$ We will choose a sufficiently large \( L \geq 1 \) in the following \eqref{eq64}, then we define:
\begin{equation}\label{eq28}
  \begin{gathered}
\widetilde{\Omega}_{h}=\left\{\frac{1}{2}-2Ch \leq \Re z \leq \frac{1}{2}+2Ch\right\} \cap\left\{\overline{q}+\frac{3+\kappa}{4}w(h) \leq \frac{\Im z}{h} \leq 4L\|q\|_{\infty} \right\} \subset \mathbb{C},\\
\Omega_{h}=\left\{\frac{1}{2}-Ch \leq \Re z \leq \frac{1}{2}+Ch\right\} \cap\left\{\overline{q}+w(h) \leq \frac{\Im z}{h} \leq 3L\|q\|_{\infty} \right\} \subset \widetilde{\Omega}_{h} .
  \end{gathered}
\end{equation}
For any $z\in \widetilde{\Omega}_{h}$, then $f((2\Re z-1)/h)\geq 1$, and for any $\rho \in p^{-1}\left(\left[\frac{1}{2}-\epsilon_{0},\frac{1}{2}+\epsilon_{0}\right]\right)$, we have
\begin{equation}\label{eq27}
\frac{\Im z}{h}-\widehat{q}^{T}(\rho)=\frac{\Im z}{h}-q^T(\rho)+f\left(\frac{2\Re z -1}{h}\right)^2\left(q^T(\rho)-\chi_{h}\left(q^{T}\right)(\rho)\right) \geq \frac{1-\kappa}{2}w(h).
\end{equation}
Now we prove the following theorem similarly to \cite{SJ1} and \cite{AN2}:
\begin{proposition}\label{invertthm}
There exist $0<h_{0}\leq 1$ and $K=K(M,q,f,C,\kappa)$ independent with $z,h$ such that for any $h<h_{0}$ and $z \in \widetilde{\Omega}_{h}$, we have $\left(z-\widetilde{\mathcal{P}}_{T}(z,h)\right)^{-1}$ exists, and we have
$$
\left\|\left(z-\widetilde{\mathcal{P}}_{T}(z,h)\right)^{-1}\right\|\leq \frac{K}{w(h)h}.
$$
\end{proposition}

\begin{proof}
  Recall that for any self-adjoint operator $A,B$ on $L^2(M)$ and any $u ,v, w \in L^2(M)$, the following identity and inequality hold:
  $$
  \|(A+i B) u\|^2=\|A u\|^2+\|B u\|^2+i\langle [A, B] u, u\rangle, \quad 2||v-w||^2\geq ||v||^2-2||w||^2.
  $$
  By \eqref{eq31}, for any $u \in H^2(M)$, we have:
  \begin{equation}\label{eq34}
    2\left\|\left(\widetilde{\mathcal{P}}_T-z\right) u\right\|^2 \geq\left\|\left(P+i h \widehat{Q}_T-z\right) u\right\|^2-\mathcal{O}\left(h^{4-4\delta}\right)\left(\left\|\left(2P-1\right) u\right\|^2+\|u\|^2\right),
  \end{equation}
  and 
  \begin{equation}\label{eq35}
    \begin{aligned}
    \left\|\left(P+i h \widehat{Q}_T-z\right) u\right\|^2=\|(P-\Re z) u\|^2+h^2\left\|\left(\frac{\Im z}{h}-\widehat{Q}_T\right) u\right\|^2+i h\left\langle \left[P, \widehat{Q}_T\right] u, u\right\rangle.
    \end{aligned}
  \end{equation}
  The final term satisfies
  \begin{equation}\label{eq55}
    \begin{aligned}
    \left[P, \widehat{Q}_T\right]&=\left[P, Q_T\right] + \left[P, f\left(\frac{2P-1}{h}\right)\mathcal{Q}_{T}f\left(\frac{2P-1}{h}\right)\right] \\
    &=\left[P, Q_T\right]+ f\left(\frac{2P-1}{h}\right)\left[P, \mathcal{Q}_{T}\right]f\left(\frac{2P-1}{h}\right).
    \end{aligned}
  \end{equation}
By Proposition \ref{prop1}, for any $a \in S_{\delta}^{1}(T^*M)$ with $0\leq \delta <\frac{1}{2}$, since $p=\sigma_{h}(P) \in S^{2}(T^*M)$, we have:
$$[P,\mathrm{Op}_{h}(a)]-\frac{h}{i}\mathrm{Op}_{h}(\{p,a\}) \in h^{2-2\delta}\Psi^{1}_{\delta}(M).$$
Combine with local form of Gårding inequality \ref{localCVineq}, for any $u\in H^2(M)$, we have
\begin{equation}\label{eq81}
\|[P,\mathrm{Op}_{h}(a)]u\|_{L^2}\leq h\left(\sup_{p^{-1}\left(\frac{1-\epsilon_{0}}{2},\frac{1+\epsilon_{0}}{2}\right)}\{p,a\}+\mathcal{O}\left(h^{1-2\delta}\right)\right)\|u\|_{L^2}+\mathcal{O}\left(h\right)\|(2P-1)u\|_{L^2}.
\end{equation}
 Recall that $q^{T}=\langle q \rangle_{T}$ on $p^{-1}\left(\left[\frac{1}{2}-\epsilon_{0},\frac{1}{2}+\epsilon_{0}\right]\right)$, and $X$ is the Hamiltonian vector field generated by geodesic flow $\varphi^{t}$,
\begin{equation*}
  \begin{aligned}
&\sup_{p^{-1}\left(\frac{1-\epsilon_{0}}{2},\frac{1+\epsilon_{0}}{2}\right)}\{p,q^{T}\}=\sup_{\rho \in p^{-1}\left(\frac{1-\epsilon_{0}}{2},\frac{1+\epsilon_{0}}{2}\right)}X\left(\langle q \rangle_{T}(\rho)\right)\\
=&\sup_{\rho \in p^{-1}\left(\frac{1-\epsilon_{0}}{2},\frac{1+\epsilon_{0}}{2}\right)}\frac{q(\varphi^{T/2}(\rho))-q(\varphi^{-T/2}(\rho))}{T}=\mathcal{O}\left(\frac{1}{T}\right).
  \end{aligned}
\end{equation*}
Similarly,
\begin{equation*}
  \begin{aligned}
&\sup_{p^{-1}\left(\frac{1-\epsilon_{0}}{2},\frac{1+\epsilon_{0}}{2}\right)}\{p,\chi_{h}\left(q^{T}\right)\}=\sup_{\rho \in p^{-1}\left(\frac{1-\epsilon_{0}}{2},\frac{1+\epsilon_{0}}{2}\right)}X\left(\chi_{h}\circ\langle q \rangle_{T}(\rho)\right)\\
=&\sup_{\rho \in p^{-1}\left(\frac{1-\epsilon_{0}}{2},\frac{1+\epsilon_{0}}{2}\right)}\chi_{h}^{\prime}\left(\langle q \rangle_{T}(\rho)\right)X\left(\langle q \rangle_{T}(\rho)\right)=\mathcal{O}\left(\frac{1}{T}\right).
  \end{aligned}
\end{equation*}
Recall that $T=(1-4\epsilon)|\log h|/\Lambda$, $\left|\Re z-\frac{1}{2}\right|=\mathcal{O}(h)$.We apply \(a = q^{T}\) and \(a = \widehat{q}^{T}\) in \eqref{eq81}; thus, we have
\begin{equation}\label{eq56}
  \begin{aligned}
    &ih\left\langle \left[P, \widehat{Q}_T\right] u, u\right\rangle=ih\left\langle \left[P, Q_{T} \right] u, u\right\rangle+ih \left\langle \left[P, \mathcal{Q}_{T}\right] f\left(\frac{2P-1}{h}\right)u, f\left(\frac{2P-1}{h}\right)u \right\rangle\\
    &=\mathcal{O}\left(h^2\right)\left(\left(\frac{1}{|\log h|}+h^{1-2\delta}\right)\left(\|u\|_{L^2}+\left\|f\left(\frac{2P-1}{h}\right)u\right\|_{L^2}\right)\right)\\
    &+\mathcal{O}\left(h^2\right) \left(\|(2P-1)u\|_{L^2}+\left\|(2P-1)f\left(\frac{2P-1}{h}\right)u\right\|_{L^2}\right)\\
    &=\mathcal{O}\left(\frac{h^2}{|\log h|}\right)\|u\|^2+\mathcal{O}(h^2)\|\left(P-\Re z\right)u\|^2.
  \end{aligned}
\end{equation}
Here we use that $\|(2P-1)f\left(\frac{2P-1}{h}\right)\|=\mathcal{O}(h)$. Combine \eqref{eq34}, \eqref{eq35}, and \eqref{eq56}, we obtain that:
\begin{equation}\label{eq36}
  \|(P-\Re z) u\|\leq 2\left\|\left(\widetilde{\mathcal{P}}_T-z\right) u\right\|+\mathcal{O}\left(\frac{h}{|\log h|^{\frac{1}{2}}}\right)\|u\|.
\end{equation}
On the other hand, we have
\begin{equation}\label{eq38}
  \Im \left\langle \frac{1}{h}\left(z-\widetilde{\mathcal{P}}_T\right) u, u\right\rangle=\left\langle\left(\frac{\Im z}{h}-\widehat{Q}_T\right) u, u\right\rangle+\mathcal{O}\left(h^{1-2 \delta}\right)(\|u\|+\|(\Re z-P) u\|)\|u\|,
\end{equation}
and by the definition of $\widehat{Q}_{T}$,
\begin{equation}\label{eq39}
\begin{aligned}
\left\langle\left(\frac{\Im z}{h}-\widehat{Q}_T\right) u, u\right\rangle=\left\langle\left(\frac{\Im z}{h}-Q_T+f\left(\frac{2 P-1}{h}\right)\mathcal{Q}_{T} f\left(\frac{2 P-1}{h}\right)\right) u, u\right\rangle.
\end{aligned}
\end{equation}

Here we would like to eliminate $f\left(\frac{2P-1}{h}\right)$ and for that purpose we factorize
$$
f(\lambda)=f(\mu)+g_\mu(\lambda)(\lambda-\mu),
$$
so that $g_\mu(\lambda)=\int_0^1 f^{\prime}(\mu+t(\lambda-\mu)) d t$ and $\left|g_\mu(\lambda)\right| \leq \left\|f^{\prime}\right\|_{L^{\infty}(\mathbb{R})}$. Let $g(\lambda):=g_{\frac{2\Re z-1}{h}}(\lambda)$, we have
$$f\left(\frac{2P-1}{h}\right) - f\left(\frac{2\Re z-1}{h}\right)=g\left(\frac{2P-1}{h}\right)\frac{2P-2\Re z}{h}.$$
Thus, we get
\begin{equation}\label{eq57}
  \begin{aligned}
    & \left|\left\langle f\left(\frac{2P-1}{h}\right)\mathcal{Q}_{T} f\left(\frac{2P-1}{h}\right) u,u\right\rangle - f\left(\frac{2\Re z-1}{h}\right)^2 \langle \mathcal{Q}_{T} u, u \rangle \right| \\
    & \leq \left|\left\langle \mathcal{Q}_{T} \left(f\left(\frac{2P-1}{h}\right) - f\left(\frac{2\Re z-1}{h}\right)\right) u, f\left(\frac{2P-1}{h}\right) u\right\rangle\right| \\
    & \quad + \left|\left\langle f\left(\frac{2\Re z-1}{h}\right)\mathcal{Q}_{T}u, \left(f\left(\frac{2P-1}{h}\right) - f\left(\frac{2\Re z-1}{h}\right)\right)u \right\rangle \right| \\
    & \leq \left|\left\langle \mathcal{Q}_{T} g\left(\frac{2P-1}{h}\right) \frac{2P-2\Re z}{h} u, f\left(\frac{2P-1}{h}\right) u\right\rangle\right| \\
    & \quad + \left|\left\langle \mathcal{Q}_{T} f\left(\frac{2\Re z-1}{h}\right) u, g\left(\frac{2P-1}{h}\right) \frac{2P-2\Re z}{h} u\right\rangle\right| \\
    & \leq \left( 2 \left(\sup _{\rho \in T^*M} \widehat{q}^{T}(\rho)\right) \|f\|_{\infty} \left\| f^{\prime} \right\|_{\infty} + \mathcal{O}(h^{1-2\delta}) \right) \|u\| \left\| \frac{P - \Re z}{h} u \right\|\\
    & =\mathcal{O}(1)\|u\| \left\| \frac{P - \Re z}{h} u \right\|.
  \end{aligned}
  \end{equation}

\begin{equation}\label{eq62}
  \begin{aligned}
  \left\langle\left(\frac{\Im z}{h}-\widehat{Q}_T\right) u, u\right\rangle=\left\langle\left(\frac{\Im z}{h}-Q_T+f\left(\frac{2 \Re z-1}{h}\right)^2\mathcal{Q}_{T}\right) u, u\right\rangle+\mathcal{O}(1)\|u\|\left\|\frac{P-\Re z}{h} u\right\|.
  \end{aligned}
  \end{equation}

Recall \eqref{eq27}, it follows from the local form of Gårding inequality Theorem \ref{localCVineq} that for such $z \in \widetilde{\Omega}_{h}$,
\begin{equation}\label{eq40}
\begin{aligned}
&\left\langle\left(\frac{\Im z}{h}-Q_T+f\left(\frac{2 \Re z-1}{h}\right)^2\mathcal{Q}_{T}\right) u, u\right\rangle\\
\geq &\left(\frac{(1-\kappa)w(h)}{2}-\mathcal{O}\left(h^{1-2 \delta}\right)\right)\|u\|^2-\mathcal{O}(1)\|u\|\|(P-\Re z) u\| .
\end{aligned}
\end{equation}
Then for sufficiently small $h$, we get
\begin{equation}\label{eq41}
\begin{aligned}
\Im \left\langle\frac{1}{h}\left(z-\widetilde{\mathcal{P}}_T\right) u, u\right\rangle \geq \frac{(1-\kappa)w(h)}{4}\|u\|^2-\mathcal{O}(1)\|u\|\left\|\frac{P-\Re z}{h} u\right\|.
\end{aligned}
\end{equation}
Since $$\Im \left\langle\frac{1}{h}\left(z-\widetilde{\mathcal{P}}_T\right) u, u\right\rangle\leq \|u\|\left\|\frac{z-\widetilde{\mathcal{P}}_T}{h} u\right\|,$$
by \eqref{eq36} and \eqref{eq41}, we have
\begin{equation}\label{eq61}
  \begin{aligned}
  \frac{(1-\kappa)w(h)}{4}\|u\| \leq \mathcal{O}(h^{-1})\left(2\left\|\left(\widetilde{\mathcal{P}}_T-z\right) u\right\|+\mathcal{O}\left(\frac{h}{|\log h|^{\frac{1}{2}}}\right)\|u\|\right).
  \end{aligned}
  \end{equation}
Recall that $\lim\limits_{h\to 0}w(h)|\log h|^{\frac{1}{2}}=\infty$. We find finally for sufficiently small $h$, $z \in \widetilde{\Omega}_{h}$, we have $w(h)\|u\|\leq \mathcal{O}(h^{-1})\left\|\left(\widetilde{\mathcal{P}}_T-z\right) u\right\|.$ It implies that the operator $z-\widetilde{\mathcal{P}}_T$ is invertible, and
$$
\left\|\left(\widetilde{\mathcal{P}}_T-z\right)^{-1}\right\|=\mathcal{O}\left(\frac{1}{w(h)h}\right).
$$
We finish the proof of Theorem \ref{invertthm}. 
\end{proof}
\section{Proof of main Theorems \ref{mainthm1}, \ref{concentrationforDWE}, and \ref{zerosoftsz}}\label{sec5}
For $z \in \widetilde{\Omega}_{h}$, we can write
$$
\mathcal{P}_T-z=\left(\widetilde{\mathcal{P}}_T-z\right)(1+K(z))
$$
where $K(z)$ is the trace class operator $\left(\widetilde{\mathcal{P}}_T-z\right)^{-1}\left(\mathcal{P}_T-\widetilde{\mathcal{P}}_T\right)$. We can bound the number of $\Sigma \bigcap \Omega_{h}$ by the number of zeros of the determiant $g(z)=\operatorname{det}(1+K(z))$, which is holomorphic in $z \in \widetilde{\Omega}_{h}$. Now we estimate the trace norm of $K(z)$, we recall that $q^{T}=\langle q \rangle_{T}$ on $p^{-1}(1/2)=S^*M$. There exists a constant $C_{1}>0$ such that:
\begin{equation}\label{eq48}
  \left\|\left( \widetilde{\mathcal{P}}_T-z \right)^{-1}\left(\mathcal{P}_T-\widetilde{\mathcal{P}}_T\right)\right\|_{\mathrm{tr}}\leq \frac{C_{1}}{w(h)h} \cdot h^{2-d} \left( \int_{S^*M} \left(q^T-\chi_{h}\left(q^{T}\right)\right) d\mu_{L}(\rho) + \mathcal{O}(h^{1-2\delta}) \right).
\end{equation}
Recall Corollary \ref{coroMDPA} of Theorem \ref{MDPA} and $T=(1-4\epsilon)|\log h|/\Lambda$, for any $$0<c<\widetilde{c}(q,M,\epsilon_{0},\epsilon,\kappa)=\frac{(1-4\epsilon)\kappa^2}{2\Lambda \sigma_{q}^2}$$ such that:
\begin{equation}\label{eq49}
  \begin{aligned}
    &\int_{S^*M} \left(q^T-\chi_{h}\left(q^{T}\right)\right) d\mu_{L}(\rho)=\int_{\langle q \rangle \geq \overline{q} + \kappa w(h)} \left(\langle q \rangle_{T}(\rho)-\chi_{h}(\langle q \rangle_{T}(\rho))\right) d\mu_{L}(\rho)\\ 
    \leq & 2\|q\|_{\infty} \mu_{L}\left\{\rho \in S^*M \ \Bigg| \ \langle q\rangle_{T}(\rho)-\overline{q}\geq \kappa w(h) \right\}\\
    =&\mathcal{O} \left(e^{-cw(h)^2|\log h|}\right).
  \end{aligned}
\end{equation}
Thus, we obtain 
\begin{equation}\label{eq50}
\|K(z)\|_{\mathrm{tr}}=\left\|\left( \widetilde{\mathcal{P}}_T-z \right)^{-1}\left(\mathcal{P}_T-\widetilde{\mathcal{P}}_T\right)\right\|_{\mathrm{tr}}=\mathcal{O}\left(\frac{h^{1-d}}{e^{cw(h)^2|\log h|}w(h)}\right).
\end{equation}
Let us call $N\left(g, \Omega_{h}\right)$ this number of zeros. We will apply Jensen's formula to bound the number of zeros $N\left(g, \Omega_{h}\right)$. However, the domain $\widetilde{\Omega}_{h}$ and $\Omega_{h}$ are both depending on $h$, so we need to deal with distance between two domains.

First we introduce $z_0=\frac{1}{2}+2 i hL\|q\|_{\infty}$. Let $$F_{h}(z)=\frac{z-1/2}{h}-i\left(\overline{q}+\frac{3+\kappa}{4}w(h)\right),$$
and $$\Omega_{0}^{\prime}:=\left\{\left|\Re z^{\prime}\right|\leq 2C, \quad 0 \leq \Im z^{\prime} \leq \frac{7}{2}L\|q\|_{\infty}\right\}.$$
Without loss of generality, we assume that $0\leq w(h)\leq \|q\|_{\infty}$, then for sufficiently large $L\geq 5$,
$$\Omega^{1}_{h}=F_{h}(\Omega_{h})=\left\{\left|\Re z^{\prime}\right|\leq C, \quad \frac{(1-\kappa)w(h)}{4} \leq \Im z^{\prime} \leq 3L\|q\|_{\infty}-\left(\overline{q}+\frac{3+\kappa}{4}w(h)\right)\right\}$$
satisfying
$$\Omega^{1}_{h}\Subset \Omega_{0}^{\prime} \Subset F_{h}(\widetilde{\Omega}_{h}),$$
and $$z_{0}^{\prime}=F_{h}(z_{0})=\left(2\|q\|_{\infty}-\overline{q}-\frac{3+\kappa}{4}w(h)\right)i \in \Omega^{1}_{h}.$$

We see that $\Omega^{1}_{h}$ is close to the edge $[-2C,2C]$ of $\Omega_{0}^{\prime}$ with distance $(1-\kappa)w(h)/4$. However, $\Omega^{1}_{h}$ is far from the corner of $\Omega_{0}^{\prime}$ with distance larger than $C>0$. Then we choose a $C^{\infty}$ Jordan curve $\gamma$ such that $\gamma$ is smooth near the corner, coincides with $\partial \Omega_{0}^{\prime}$ except the a sufficiently neighbourhood of the corner, and $\gamma$ encloses a fix simply connected domain $\Omega_{0}$ which is independent with $h$ such that $\Omega^{1}_{h} \Subset \Omega_{0}$, see Figure \ref{figure1}.

\begin{figure}
  \center
\includegraphics[scale=0.6]{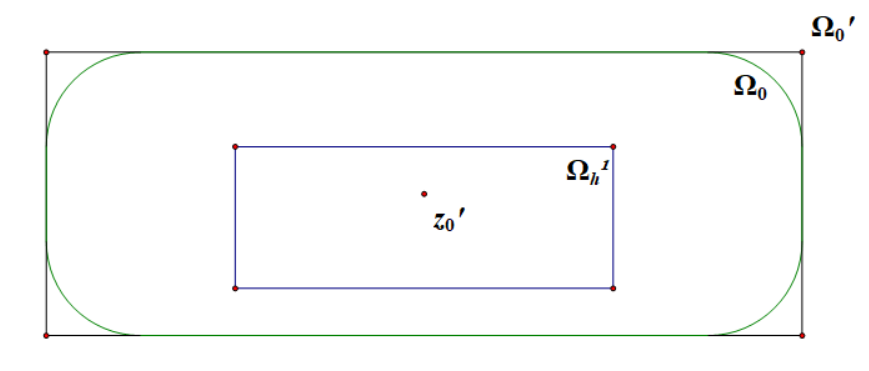}
\caption{The blue line is the boundary of domain $\Omega_{h}^{1}$, the green line is the smooth boundary of domain $\Omega_{0}$, the black line is the boundary of domain $\Omega_{0}^{\prime}$.}
\label{figure1}
\end{figure}

Since $\Omega_{0}$ is bounded simply connect domain with smooth boundary $\gamma=\partial \Omega_{0}$ and $\gamma$ is smooth diffeomorphic to $S^1=\partial \mathbb{D}(0,1)$, by Riemann mapping Theorem and Painlevé's Theorem, see \cite{BS90}, there exists a biholomorphic map $\Phi: \Omega_{0}\mapsto \mathbb{D}(0,1)$ which can be $C^\infty$ extended to $\overline{\Omega_{0}}$, and $d\Phi$ is nowhere vanishing on $\overline{\Omega_{0}}$. Let $\phi_{h}: \mathbb{D}(0,1) \mapsto \mathbb{D}(0,1)$ be the biholomorphic map such that $$ \phi_{h}(z)=\frac{z-\Phi(z^{\prime}_{0})}{1-\overline{\Phi(z^{\prime}_{0})}z}.$$ We notice that there exists $c_{0}>0$ independent with $h$ such that $d(\Phi(z^{\prime}_{0}), \partial \mathbb{D}(0,1))>c_{0}$. Then $\Phi_{h}=\phi_{h}\circ \Phi$ satisfies uniform bound of derivative, i.e. there exists $0<c_{1}<1$ such that $$c_{1}<|\Phi_{h}^{\prime}(z)|<c_{1}^{-1}, \quad z \in \overline{D(0,1)}.$$ Let $\Phi_{h}(z_1)$ belongs to the boundary of $\Phi_{h}(\Omega_{h}^{1})$, $\partial (\Phi_{h}(\Omega_{h}^{1}))$, $\Phi_{h}(z_2) \in \partial \mathbb{D}(0,1)$, then $$d(\Phi_{h}(z_{1}),\Phi_{h}(z_{2}))\geq c_{1}d(z_{1},z_{2})\geq \frac{(1-\kappa)c_{1}w(h)}{4}.$$ Then we have $$\Phi(\Omega^{1}_{h}) \subset \mathbb{D}\left(0,1-\frac{(1-\kappa)c_{1}w(h)}{4}\right).$$ We consider the $\widetilde{g}=g\circ (\Phi_{h} \circ F_{h})^{-1}$ on $\overline{\mathbb{D}(0,1)}$, then it is holomorhic on $\mathbb{D}(0,1)$, by Jensen's formula:
$$
\int_0^1 \frac{N(t)}{t} d t=\frac{1}{2 \pi} \int_0^{2 \pi} \log \left|\widetilde{g}\left(e^{i \theta}\right)\right| d \theta-\log |\widetilde{g}(0)|.
$$
where $N(t)$ denotes the number of zeros of $\widetilde{g}$ in the disc of radius $t$ centered at the origin. Let $t(h)=1-\frac{(1-\kappa)c_{1}w(h)}{4}$, then we have:
$$
N(t(h))\leq \frac{\int_{t(h)}^1 \frac{N(t)}{t} d t}{\int_{t(h)}^{1}\frac{dt}{t}}\leq \frac{1}{\int_{t(h)}^{1}\frac{dt}{t}}\left(\log ||\widetilde{g}||_{\infty,\overline{\mathbb{D}(0,1)}}-\log |\widetilde{g}(0)|\right).
$$
It implies 
$$
N(t(h))=\mathcal{O}\left(\frac{1}{w(h)}\right)\left( \log ||\widetilde{g}||_{\infty,\overline{\mathbb{D}(0,1)}}-\log |\widetilde{g}(0)|\right).
$$
Then we obtain that :
\begin{equation}\label{eq60}
N\left(g, \Omega_{h}\right) \leq N(t(h))=\mathcal{O}\left(\frac{1}{w(h)}\right)\left(\log \|g\|_{\infty, \widetilde{\Omega}_{h}}-\log \left|g\left(z_0\right)\right|\right).
\end{equation}
On the one hand, for all $z \in \widetilde{\Omega}_{h}$, by \eqref{eq50}, we have:
\begin{equation}\label{eq51}
\begin{aligned}
\log |g(z)|=\log |\operatorname{det}(1+K(z))| \leq \|K(z)\|_{\mathrm{tr}}=\mathcal{O}\left(\frac{h^{1-d}}{e^{cw(h)^2|\log h|}w(h)}\right).
\end{aligned}
\end{equation}

On the other hand, by the same argument in the proof of Theorem \ref{invertthm}, since $z_{0}$ has large imaginary part, $\mathcal{P}_T-z_{0}$ is invertible and $$\|\left(\mathcal{P}_T-z_{0}\right)^{-1}\|=\mathcal{O}\left(\frac{1}{hL\|q\|_{\infty}}\right).$$ 
Then we choose a sufficiently large $L\gg 5\|f\|_{L^{\infty}(\mathbb{R})}^2$ such that,
\begin{equation}\label{eq64}
\left\|h\left(\mathcal{P}_T-z_{0}\right)^{-1}\left(f\left(\frac{2P-1}{h}\right)\mathcal{Q}_{T}f\left(\frac{2P-1}{h}\right)\right)\right\|<\frac{1}{2}.
\end{equation}
By Neumann series, it implies that
\begin{equation}\label{eq63}
  \begin{aligned}
&\|\left(1+K\left(z_0\right)\right)^{-1}\|=\left\|\left(\mathcal{P}_T-z_{0}\right)^{-1}\left(\widetilde{\mathcal{P}}_T-z_{0}\right)\right\|\\
=&\left\|I+ih\left(\mathcal{P}_T-z_{0}\right)^{-1}\left(f\left(\frac{2P-1}{h}\right)\mathcal{Q}_{T}f\left(\frac{2P-1}{h}\right)\right)\right\|<2.
  \end{aligned}
\end{equation}
We use the same calculation as in \cite{SJ1, AN2} and get
\begin{equation}\label{eq52}
  \begin{aligned}
    &-\log |g(z_{0})|=\log \left|\operatorname{det}\left(1+K\left(z_0\right)\right)^{-1}\right|=\log \left|\operatorname{det}\left(1-K(z_0)(1+K(z_0))^{-1}\right)\right|\\
    \leq &\left\|K\left(z_0\right)\left(1+K\left(z_0\right)\right)^{-1}\right\|_{\mathrm{tr}}\leq \left\|K\left(z_0\right)\right\|_{\mathrm{tr}}\left\|\left(1+K\left(z_0\right)\right)^{-1}\right\|=\mathcal{O}\left(\frac{h^{1-d}}{e^{cw(h)^2|\log h|}w(h)}\right).
  \end{aligned}
  \end{equation}
Combining \eqref{eq60}, \eqref{eq51}, \eqref{eq52}, then we obtain that
$$
N\left(g, \Omega_{h}\right)=\mathcal{O}\left(\frac{h^{1-d}}{e^{cw(h)^2|\log h|}w(h)^2}\right).
$$
Since \( N(g, \Omega_{h}) \) equals the number of poles of \( (z - P(z, h))^{-1} \) in \( \Omega_{h} \), and we define  
\[
c(q, M) := \lim_{\kappa \to 1, \epsilon_{0} \to 0, \epsilon \to 0} \frac{(1 - 4\epsilon)\kappa^2}{2\Lambda \sigma_{q}^2} = \frac{1}{2\Lambda_{0}\sigma_{q}^2},
\]  
we conclude the proof of Theorem \ref{mainthm1}.

Finally, we prove Theorems \ref{concentrationforDWE} and \ref{zerosoftsz}. Since Theorem \ref{zerosoftsz} is equivalent to \eqref{eq59}, we only need to prove \eqref{eq59}. First, we prove Theorems \ref{concentrationforDWE}. Let $\lambda=h^{-1}\gg 1$, $\tau_{n}$ be the eigenvalue of the operator \eqref{DWEop} with $\Re \tau_{n} \sim \lambda=h^{-1}$, then $z_{n}=h^{2}\tau_{n}^2/2$ is the pole of $P(z,h)=-\frac{h^2\Delta}{2}-ih\sqrt{2z}a(x)$. For any $C_{1}>0$, if $\tau_{n}$ satisfies
$$\lambda- C_{1}\leq \Re \tau_{n} \leq \lambda+C_{1}, \quad |\Im \tau_{n} +\overline{a}|\geq |\log \lambda|^{-\frac{1-\alpha}{2}}.$$
For any $c>0$, then we take a sufficiently small $\varepsilon>0$ such that $\frac{c}{(1-\varepsilon)^2}<c(a,M)$, there exists $C>0$ depending on $a$ and $C_{1}$ such that for sufficiently large $\lambda=h^{-1}>0$, 
$$\frac{1}{2}- Ch\leq \Re z_{n} \leq \frac{1}{2}+Ch, \quad \left|\frac{\Im z_{n}}{h} +\overline{a}\right|\geq (1-\varepsilon)|\log h|^{-\frac{1-\alpha}{2}}.$$

We take $q_{z}(x,\xi)=-\sqrt{2z}a(x)$ and $w(h)=(1-\varepsilon)|\log h|^{-\frac{1-\alpha}{2}}=(1-\varepsilon)|\log \lambda|^{-\frac{1-\alpha}{2}}$, then $\frac{c}{(1-\varepsilon)^2}w(h)^2|\log h|=c|\log h|^{\alpha}$, and $q(x,\xi)$ equals to $-a(x)$, it implies $\overline{q}=-\overline{a}$. By Theorem \ref{mainthm1}, there exists $\lambda_{0}>0$ and $C_{2}>0$ such that for any $\lambda\geq \lambda_{0}$, we have 
$$
  \#\left\{\tau_{n} \ | \ \lambda- C_{1}\leq \Re \tau_{n} \leq \lambda+C_{1}, \ |\Im \tau_{n} +\overline{a}|\geq |\log \lambda|^{-\frac{1-\alpha}{2}}\right\}\leq \frac{C_{2}\lambda^{d-1}}{e^{c|\log \lambda|^{\alpha}} |\log  \lambda|^{\alpha-1}}.
$$
Let 
$$N_{k}:=\#\left\{\tau_{n} \ | \ \lambda_{0}+(k-1)C_{1}\leq \Re \tau_{n} \leq \lambda_{0}+(k+1)C_{1}, \ |\Im \tau_{n} +\overline{a}|\geq |\log (\lambda_{0}+kC_{2})|^{-\frac{1-\alpha}{2}}\right\}.$$
We notice that $f(\lambda)=\frac{\lambda^{d-1}}{e^{c|\log \lambda|^{\alpha}} |\log  \lambda|^{\alpha-1}}$ is a monotone increasing function for $\lambda \gg 1$, then 
$\sum_{k=0}^{N}f(k)\leq \mathcal{O}\left(Nf(N)\right)$ as $N \to \infty$, thus we have
\begin{equation}
  \begin{aligned}
    &\#\left\{n \ | \ 0 \leq \Re \tau_n \leq \lambda, \ \left|\Im \tau_n+\overline{a}\right|\geq (\log \lambda)^{-\frac{1-\alpha}{2}} \right\}=\sum\limits_{0\leq k\leq \frac{\lambda-\lambda_{0}}{2C_{1}}} N_{k}+\mathcal{O}(1)\\
    =& \mathcal{O}\left(\frac{\lambda^{d}}{e^{c|\log \lambda|^{\alpha}} |\log  \lambda|^{\alpha-1}}\right), \quad \lambda \to \infty.
  \end{aligned}
\end{equation}
We have proved Theorem \ref{concentrationforDWE}. For proof of \eqref{eq59}, let $P_{\omega}=-\frac{1}{2}h^2\Delta_{\omega}=P+ihQ$, where $$\quad P=\frac{1}{2}\left(-h^2\Delta+2h\langle ih\Im \omega,d\cdot \rangle -h^2|\omega|_{x}^2\right), \quad Q:=\langle -ih\Re \omega, d\cdot\rangle.$$ Then $$\sigma_{h}(P)=\frac{1}{2}|\xi|_{x}^2, \quad q(x,\xi):=\sigma_{h}(Q)(x,\xi)=\Re\langle \omega,\xi\rangle_{x}, \quad \overline{q}=\int_{S^*M}q(\rho)d\mu_{L}(\rho)=0.$$ Thus we apply Theorem \ref{mainthm1} and the same argument in the proof of Theorem \ref{concentrationforDWE} to obtain \eqref{eq59}. We have proved Theorem \ref{zerosoftsz}.

\section{Appendix: Proof of moderate deviation principles Theorem \ref{MDPA}}\label{appendix}

Here we give the proof of Theorem \ref{MDPA} in this appendix. The method of the proof of the MDP Theorem \ref{MDPA} can be found in Rey-Bellet and Young \cite[Theorem 4.3]{RY08}, Wu \cite[Theorem 1.2, Theorem 1.4]{WL95}, and Dolgopyat and Sarig \cite{DS23}. The moderate deviation is an application of Gärtner-Ellis theorem \cite{GJ77, ER84}, here we introduce a 1-dimension version, see \cite[Theorem A.1]{DS23}. 
\begin{definition}
  The Legendre-Fenchel transform of a convex function $\varphi: \mathbb{R} \rightarrow \mathbb{R}$ is the function $\varphi^*: \mathbb{R} \rightarrow \mathbb{R} \cup\{+\infty\}$ given by

$$
\varphi^*(\eta):=\sup _{\xi \in \mathbb{R}}\{\xi \eta-\varphi(\xi)\}.
$$

\end{definition}
\begin{theorem}[Gärtner-Ellis theorem]\label{GETLDP}
  Suppose $t>0$, $\lim\limits_{t\to \infty}b(t)=\infty$, and let $W_{t}$ be a family of random variables such that $\mathbb{E}\left(\mathrm{e}^{\xi W_t}\right)<\infty$ for all $\xi \in \mathbb{R}$. Assume that the limit

  $$
  \mathcal{F}(\xi):=\lim _{t \rightarrow \infty} \frac{1}{b(t)} \log \mathbb{E}\left(\mathrm{e}^{\xi W_{t}}\right)
  $$
  
  \noindent exists for all $\xi \in \mathbb{R}$, and is differentiable and strictly convex on $\mathbb{R}$. Let $\mathcal{I}(\eta)$ be the Legendre-Fenchel transform of $\mathcal{F}(\xi)$. Then:
  \begin{itemize}
  \item For every closed set $F \subset \mathbb{R}$, $\limsup\limits _{t \rightarrow \infty} \frac{1}{b(t)} \log \mathbb{P}\left[W_{t} / b(t) \in F\right] \leq-\inf\limits_{\eta \in F} \mathcal{I}(\eta)$.
  \item For every open set $G \subset \mathbb{R}$, $\liminf\limits _{t \rightarrow \infty} \frac{1}{b(t)} \log \mathbb{P}\left[W_{t} / b(t) \in G\right] \geq-\inf\limits_{\eta \in G} \mathcal{I}(\eta)$.
  \end{itemize}
  \end{theorem}
  In our case, $(S^*M,\mu_{L})$ is a probability space, for $t>0$, let $q_{t}(\rho)=\int_{0}^{t}q \circ \varphi^{s}ds-t\overline{q}$ be a family of random variables 
   and we define
  $$
  Z_{t}(\xi):=\mathbb{E}\left(\mathrm{e}^{\xi q_t}\right)= \int_{S^*M}e^{\xi q_{t}(\rho)}d\mu_{L}(\rho)<\infty, \quad F_{t}(\xi)=\frac{1}{t}\log Z_{t}(\xi).
  $$
  In the proof of LDP, see Kifer \cite{KY}, and Young \cite{LSY}, show that for any $\xi \in \mathbb{R}$, $\lim\limits_{t \to \infty} F_{t}(\xi)$ exists, we denote it by $F(\xi)$, and we have:
  $$F(\xi):=\lim\limits_{t \to \infty}F_{t}(\xi)=\mathrm{Pr}(\xi(q-\overline{q})).$$
  In the following, we will show that for a sufficiently large $t_{0}>0$, $\{F_{t}(\xi)\}_{t\geq t_{0}}$ is a normal family of holomorphic functions on a neighbourhood of $\xi$.
  \begin{proposition}\label{normalfamily}
  There exists $\delta>0$, $t_{0}>0$ and $C>0$ such that for any $t\geq t_{0}$ and $|\xi|\leq \delta$, we have
  \begin{equation}
  |F_{t}(\xi)|\leq C.
  \end{equation}
  It implies that $F(\xi)$ has an analytic continuation on $D(0,\delta)$, and for any $n$-th derivative, $\lim\limits_{t \to \infty} F^{(n)}_{t}(\xi) \to F^{(n)}(\xi)$ uniformly on any compact subset of $D(0,\delta)$.
  \end{proposition}
  First, we introduce the Tsujii's theorem \cite[Theorem 1.1]{TM10} which showed that the quasi-compactness of transfer operators for contact Anosov flows, and the geodesic flow on cosphere bundle is contact Anosov flow.
  \begin{theorem}[Tsujii, \cite{TM10}]\label{spectralgapfortransfer}
    We assume that $M$ is a compact Anosov manifold, for each $\beta>0$, there exists a Hilbert space $B^\beta$ with inner product $\left(-,-\right)_{B^{\beta}}$, which contains the Sobolev space $H^s(S^*M)$ and is contained in the dual space $H^{-s}(S^*M)$ for $s>\beta$, such that the transfer operator $\mathcal{L}^t$ which is defined as
    \begin{equation}\label{transfer}
    \mathcal{L}^{t}f:=f\circ \varphi^{t}, \quad f\in C^{\infty}(S^*M,\mathbb{C}).
    \end{equation}
    for sufficiently large $t$ extends to a bounded operator on $B^\beta$ and the essential spectral radius of the extension $\mathcal{L}^t: B^\beta \rightarrow B^\beta$ is bounded by $e^{-\sigma t}$, where $\sigma>0$.
    \end{theorem}

  We notice that for any $s>\beta$, $H^{s}\subset B^{\beta}\subset H^{-s}$, then for any $f \in B^{\beta}$, we have $\left|\int_{S^*M}f(x)d\mu_{L}(x)\right|\leq \|f\|_{H^{-s}}\| 1 \|_{H^{s}}$. Therefore, there exists a constant $C_{s,\beta}>0$ such that 
  \begin{equation}\label{eq76}
    \left|\int_{S^*M}f(x)d\mu_{L}(x)\right|\leq C_{s,\beta}\|f\|_{B^{\beta}}.
  \end{equation}
  Here the essential spectral radius of a bounded operator $\mathcal{L}$ on a Banach space $\mathcal{B}$ is the infimum of the real numbers $\rho>0$ so that, outside of the disc of radius $\rho$, the spectrum of $\mathcal{L}$ on $\mathcal{B}$ consists of isolated eigenvalues of finite multiplicity. 
  
  Recall that $\lambda=1$ is the simple eigenvalue of $\mathcal{L}^{t}$ by the ergodicity of $\varphi^{t}$, and there is no eigenvalue on the region $\{\lambda \ | \ |\lambda|\geq 1\}\setminus \{1\}$ of $\mathcal{L}^{t}$ by the mixing of $\varphi^{t}$. Let $\Pi$ be the orthogonal projection on $B^{\beta}$ such that $\Pi f=(f,1)_{B^{\beta}}$. Thus, we obtain that there exists a $t_{1}>0$ and $0<\sigma_{1}<1$ such that for any $f \in B^{\beta}$ and $t\geq t_{1}$, the spectral radius of $\mathcal{L}^{t_{1}}-\Pi$ satisfies $$\rho\left(\mathcal{L}^{t_{1}}-\Pi\right)<\sigma_{1}.$$ 
  We recall that the spectral radius of bounded operator $\mathcal{L}$ is defined as 
  $$\rho(\mathcal{L})=\lim\limits_{n\to \infty}\|\mathcal{L}^{n}\|^{\frac{1}{n}}.$$
  Now we prove Proposition \ref{normalfamily}.
  \begin{proof}
  Since $q_{t} \in C^{\infty}(S^*M) \subset B^{\beta}$, the multiplication operator $M_{q_{t}}f:=q_{t}f$ on $B^{\beta}$ is bounded operator. We define a strongly continuous group $\mathcal{L}^{t}_{\xi}(f)=e^{\xi q_{t}}f \circ \varphi^{t}$ on $B^{\beta}$. This family of bounded operators is holomorphically depending on $\xi \in \mathbb{C}$. Now we consider the $t=t_{1}$, then by Kato's theorems \cite[Chapter 7, Theorem 1.7 and 1.8]{KT66}, for sufficiently small $|\xi|$, $\mathcal{L}^{t_{1}}_{\xi}$ has the simple eigenvalue $\lambda(\xi)$ with largest modulus and $\Pi_{\xi}$ is the projection to the eigenfunction associated with $\lambda(\xi)$ of $\mathcal{L}^{t_{1}}_{\xi}$. We notice that $\lambda(\xi)$ and $\Pi_{\xi}$ are both continuous on a neighbourhood of $0$. Then there exist $\delta_{0}>0$ and $\epsilon>0$ such that for any $\xi \in \overline{D(0,\delta_{0})}$,
  \begin{equation}\label{eq78}
  \rho\left( \mathcal{L}^{t_{0}}_{\xi}(I-\Pi_{\xi}) \right)<\sigma_{1}<1-2\epsilon<1-\epsilon<|\lambda(\xi)|<1+\epsilon, \quad \|\Pi_{\xi}-\Pi\|<\epsilon.
  \end{equation}

  Recall that $\Pi_{\xi}$ commutes with $\mathcal{L}^{t_{1}}_{\xi}$, then we have
  \begin{equation}\label{eq77}
  \left(\mathcal{L}^{t_{1}}_{\xi}(I-\Pi_{\xi})\right)^n=\mathcal{L}^{nt_{1}}_{\xi}(I-\Pi_{\xi})=\mathcal{L}^{nt_{1}}_{\xi}-\lambda(\xi)^{n}\Pi_{\xi}.
  \end{equation}
  For any $t\geq t_{1}$, let $n:=[t/t_{1}]$, and $r_{1}:=t-nt_{1} \in [0,t_{1}]$ and we define $$f(\rho)=f_{t,\xi}(\rho)=\mathcal{L}^{r_{1}}_{\xi}(1)(\rho)=e^{\xi q_{r_{1}}(\rho)}, \quad \rho \in S^*M.$$ Then we have
  $$Z_{t}(\xi)=\int_{S^*M} \mathcal{L}^{t}_{\xi}(1)(\rho)d\mu_{L}(\rho)=\int_{S^*M} \mathcal{L}^{nt_{0}}_{\xi}(f)(\rho)d\mu_{L}(\rho).$$
  By \eqref{eq76} and \eqref{eq77}, we have
  \begin{equation}\label{eq79}
  \left|\int_{S^*M} \left(\mathcal{L}^{t_{1}}_{\xi}(I-\Pi_{\xi})\right)^{n}(f)(\rho)d\mu_{L}(\rho)\right|\leq C_{s,\beta}\|\left(\mathcal{L}^{t_{1}}_{\xi}(I-\Pi_{\xi})\right)^{n}f\|_{B^{\beta}}.
  \end{equation}
  For sufficiently large $n_{1}$, for any $n\geq n_{1}$, by \eqref{eq78}, we have
  $$\|\left(\mathcal{L}^{t_{1}}_{\xi}(I-\Pi_{\xi})\right)^{n}\|^{\frac{1}{n}}<\sigma_{1}+\epsilon.$$
  We notice that there exists $C>0$ such that for any $|\xi|\leq \delta_{0}$ and $t\geq t_{0}$, $\|f\|_{B^{\beta}}\leq C$. By \eqref{eq77} and \eqref{eq79}, we have
  \begin{equation*}
  \left|Z_{t}(\xi)-\lambda(\xi)^{n}\int_{S^*M}\Pi_{\xi}(f)(\rho)d\mu_{L}(\rho)\right|\leq C_{s,\beta}C(\sigma_{1}+\epsilon)^{n}.
  \end{equation*}
  There exists a $\delta \in (0,\delta_{0})$ sufficiently small, such that for any $|\xi|\leq \delta$,
  $$\left|\int_{S^*M}\Pi_{\xi}(f)(\rho)d\mu_{L}(\rho)-1\right|<\epsilon.$$
  Thus we have
  \begin{equation}\label{eq80}
  \frac{1}{t}\log |Z_{t}(\xi)|-\frac{[t/t_{1}]}{t}\log |\lambda (\xi)|=\mathcal{O}\left(\frac{1}{t}\right), \quad \text{uniformly for } \ |\xi|\leq \delta.
  \end{equation}
  Since $D(0,\delta)$ is a simply connected region, we choose a principal value of $\log Z_{t}(\xi)$ such that it is a holomorphically on $D(0,\delta)$, and $|\log Z_{t}(\xi)|\leq |\log |Z_{t}(\xi)||+2\pi$. By Montel's theorem, $\{F_{t}(\xi)=\frac{1}{t}\log Z_{t}(\xi)\}_{t}$ is a normal family, then it has a subsequence which uniformly convergent on each compact subset of $D(0,\delta)$. By \eqref{eq80}, for $\xi \in \mathbb{R}$, we have $\lim\limits_{t \to \infty}F_{t}(\xi) \to \log |\lambda(\xi)|/t_{1}.$ It implies that for any limit function $g(\xi)$ of the normal family $\{F_{t}(\xi)\}_{t\geq t_{0}}$, $g(\xi)=\log |\lambda(\xi)|/t_{1}$. So we prove that $F_{t}(\xi)$ is uniformly convergent to $g$ as $t\to \infty$, then also their any order derivatives. We obtain that $g(\xi)=F(\xi)=\mathrm{Pr}(\xi(q-\overline{q}))$ on $D(0,\delta)$, thus we prove that for any $n\geq 0$, $$F_{t}^{(n)}(\xi) \to F(\xi), \quad \text{uniformly for } \ t\geq t_{0}, \ |\xi|\leq \delta.$$
  \end{proof}
  Now we prove Theorem \ref{MDPA}. Let $a(t)$ be a family of positive real numbers such that $\lim\limits_{t \rightarrow \infty} a(t)=\infty$ and $\lim\limits_{t \rightarrow \infty} a(t)t^{-\frac{1}{2}}=0$. Then we define $b(t)=\frac{t}{a(t)^2}$, and $W_{t}=\frac{q_{t}}{a(t)}$,
  $$\mathcal{F}_{t}(\xi)=\frac{1}{b(t)}\log \mathbb{E}(\xi W_{t})=\frac{a(t)^2}{t}\log \mathbb{E}\left(\frac{\xi}{a(t)} q_{t}\right)=a(t)^2F_{t}\left(\frac{\xi}{a(t)}\right).$$ 
  For any $\eta \in [-\delta/2,\delta/2]$, there exists a $\eta^{\prime} \in \mathbb{R}$ with $|\eta^{\prime}|\leq \eta \leq \delta/2$, we have $$F_{t}(\eta)=F_{t}(0)+F_{t}^{(1)}(t)(0)\eta+\frac{1}{2}F_{t}^{(2)}(0)\eta^2+\frac{1}{6}F_{t}^{(3)}(\eta^{\prime})\eta^3.$$ 
  We note that for any $|\eta|\leq \delta/2$ and $t\geq t_{0}$, there exists a $C^{\prime}>0$ such that $|F^{(3)}_{t}(\eta)|\leq C^{\prime}$ by the uniformly convergence of $\{F^{(3)}_{t}\}$. By the definition of $F_{t}(\xi)$, we have $F_{t}(0)=0$, $F_{t}^{(1)}(0)=0$. Thus for any $R>0$, we take $t^{\prime} \gg t_{0}$ such that $R/a(t^{\prime}) \leq \delta/2$, for any $|\xi|\leq R$, $t\geq t^{\prime}$, as $t \to \infty$, we have
  $$\mathcal{F}_{t}(\xi)=a(t)^2F_{t}\left(\frac{\xi}{a(t)}\right)=\frac{1}{2}F_{t}^{(2)}(0)\xi^2+\mathcal{O}_{R}\left(\frac{1}{a(t)}\right).$$
  So we prove the assumption in the Gärtner-Ellis Theorem \ref{GETLDP}, i.e. for any $\xi \in \mathbb{R}$, we have 
  $$\mathcal{F}(\xi):=\lim_{t\to \infty}\mathcal{F}_{t}(\xi)=\frac{1}{2}F^{(2)}(0)\xi^2.$$
  By the definition of $F_{t}(\xi)$, we have
  $$F_{t}^{(2)}(0)=\frac{1}{t}\int_{S^*M}q_{t}^{2}(\rho)d\mu_{L}(\rho).$$
  Since we assume that $q$ is not cohomologous to $\overline{q}$, by \eqref{limitofvariance}, we have $F^{(2)}(0)=\sigma_{q}^2>0$. So we find that $\mathcal{F}(\xi)=\frac{1}{2}\sigma_{q}^2\xi^2$ is a smooth and strcitly convex function. Then the Legendre-Fenchel transform of $\mathcal{F}(\xi)$ is 
  $$\mathcal{I}(\eta)=\sup_{\xi \in \mathbb{R}}\left\{\xi\eta-\frac{1}{2}\sigma_{q}^{2}\xi^2\right\}=\frac{\eta^2}{2\sigma_{q}^2}.$$
  By Gärtner-Ellis Theorem \ref{GETLDP}, we prove the Theorem \ref{MDPA}.

\bibliographystyle{amsalpha}
\bibliography{article.bib}
  
\end{document}